\documentclass[11pt]{article}

\usepackage[reqno]{amsmath}
\usepackage{amssymb}
\usepackage{amsthm}
\usepackage{color}
\usepackage{dutchcal}

\usepackage{enumitem}
\usepackage{cleveref}

\usepackage{relsize}

\numberwithin{equation}{section}

\newtheorem{theorem}{Theorem}[section]
\newtheorem{lemma}[theorem]{Lemma}

\newtheorem{proposition}[theorem]{Proposition}

\newcommand{\R}{\mathbb{R}}
\newcommand{\N}{\mathbb{N}}

\newcommand{\B}{\mathbb{B}}

\makeatletter
\newcommand{\customlabel}[2]{%
	\protected@write \@auxout {}{\string \newlabel {#1}{{#2}{}}}}
\makeatother

\makeatletter
\newcommand{\leqnomode}{\tagsleft@true\let\veqno\@@leqno}
\newcommand{\reqnomode}{\tagsleft@false\let\veqno\@@eqno}
\makeatother

\begin{document}

\title{Necessary Optimality Conditions For Average Cost Minimization Problems}

\author{ 
    Piernicola Bettiol\footnote{ {\it Laboratoire de Math\'ematiques, Universit\'e de Bretagne
		Occidentale, 6 Avenue Victor Le Gorgeu, 29200 Brest, France, e-mail: \/}
						{\tt piernicola.bettiol@univ-brest.fr}} ,
     Nathalie Khalil\footnote{ {\it MODAL'X, Universit\'e Paris Ouest Nanterre La D\'efense, 200 Avenue de la R\'epublique, 92001 Paris Nanterre, France, e-mail: \/} 
            {\tt khalil.t.nathalie@gmail.com}}
}

\date{}

\maketitle
\begin{abstract}\noindent
Control systems involving unknown parameters appear a natural framework for applications in which the model design has to take into account various uncertainties. In these circumstances the performance criterion can be given in terms of an {\it average cost}, providing a  paradigm which differs from the more traditional minimax or robust optimization criteria. In this paper, we provide necessary optimality conditions for a nonrestrictive class of optimal control problems in which unknown parameters intervene in the dynamics, the cost function and the right end-point constraint. An important feature of our results is that we allow the unknown parameters belonging to a mere complete separable metric space (not necessarily compact).

\end{abstract}

\section{Introduction}
In this paper we consider a class of optimal control problems in which uncertainties appear in the data in terms of unknown parameters belonging to a given metric space. Though the state evolution is governed by a deterministic control system and the initial datum is fixed (and well-known), the description of the dynamics depends on uncertain parameters which intervene also in the cost function and the right end-point constraint. Taking into consideration an {\it average cost} criterion, a crucial issue is clearly to be able to characterize optimal controls independently of the unknown parameter action: this allows to find a sort of `best trade-off' among all the possible realizations of the control system as the parameter varies. 
In this context we provide, under non-restrictive assumptions, necessary optimality conditions.
More precisely, we consider the following average cost minimization problem:\leqnomode
\begin{equation*} \tag{P} \begin{cases}
\begin{aligned} \label{intprob}
& {\text{minimize}}
& & J_{\Omega}((u(.), \{x(.,\omega)\})) := \int_{\Omega} g( x(T,\omega); \omega)\ d \mu(\omega) \\
&&& \hspace{-1.9cm}\text{over measurable functions } u : [0,T] \rightarrow \R^m  \text{ and $W^{1,1}$ arcs } 
\{ x(.,\omega) : [0,T] \rightarrow \R^n \ | \ \omega \in \Omega   \} \\
&&& \hspace{-1.9cm}\text{ such that} \;\; \;\;\;  u(t) \in U(t) \quad \textrm{a.e. } t \in [0,T] \\
&&& \hspace{-1.9cm}\text{ and, for each} \ \omega \in \Omega, 
\\
& & & \dot{x}(t,\omega) = f(t,x(t,\omega), u(t), \omega) \quad \textrm{a.e. } t \in [0,T], \\
&&& x(0,\omega) = x_0 \quad \text{and} \quad \int_{\Omega} d_{C(\omega)} (x(T,\omega)) \ d \mu(\omega) = 0. 
\end{aligned}  	\end{cases} \end{equation*} \reqnomode
Here, $d_C(x)$ is the Euclidean distance of a point $x$ from the set $C$.
The data for this problem comprise a time interval $[0,T]$, a probability measure $\mu$ defined on a metric space $\Omega$, functions $g: \R^n \times \Omega \rightarrow \R$ and $f: [0,T] \times \R^n \times \R^m\times \Omega \rightarrow \R^n$, a nonempty multifunction $U: [0,T] \leadsto \R^m$, and a family of closed sets $ \{ C(\omega) \subset \R^n  \ | \ \omega \in \Omega \}$. 
A measurable function $u:[0,T] \rightarrow \R^m$ that satisfies
\[u(t) \in U(t) \quad \textrm{a.e. } t \in [0,T]\] is called a {\it control function}. The set of all control functions is written $\mathcal U$. A {\it process} $(u,\{x(.,\omega): \omega \in \Omega\})$ is a control function $u$ coupled with a family of arcs $\{x(.,\omega) \in W^{1,1}([0,T], \R^n): \omega \in \Omega\}$, satisfying, for each $\omega \in \Omega$, the dynamic constraint:
\[  \dot{x}(t,\omega) = f(t,x(t,\omega), u(t), \omega) \quad \text{a.e. } t\in [0,T] , \qquad x(0,\omega)=x_0. \]
A process is said to be {\it feasible} if, in addition, the arcs $x(.,\omega)$'s satisfy the averaged right end-point constraint
\[ \int_{\Omega}  d_{C(\omega)} (x(T,\omega)) \ d\mu(\omega) = 0 \  .\]
If the integral cost term in (\ref{intprob}) does not exist for a feasible process $( u,\{ x(.,\omega): \omega \in \Omega \}),$ then we set $J_{\Omega}(u(.), \{x(.,\omega)\}) = + \infty$. To underline the dependence on a given control $u(.) \in \mathcal{U}$, sometimes we shall employ the notation $x(.,u,\omega)$ for the feasible arc belonging to the family of trajectories $\{ x(.,\omega) \ : \ \omega \in \Omega  \}$, associated with the control $u(.)$ and the element $\omega \in \Omega$.\\
A feasible process $(\bar u,\{\bar x(.,\omega): \omega \in \Omega \})$ is said to be a {\it $W^{1,1}-$local minimizer} for (\ref{intprob}) if there exists $\epsilon>0$ such that
\[ \int_{\Omega} g( \bar x(T,\omega); \omega)\ d \mu(\omega)  \leq \int_{\Omega} g( x(T,\omega); \omega)\ d \mu(\omega)\] for all feasible processes $(u,\{x(.,\omega): \omega \in \Omega\})$ such that
\begin{equation}  \label{definition_norm_epsilon_minimizer} \| \bar x(.,\omega) - x(.,\omega) \|_{W^{1,1}} \leq \epsilon \quad \text{for all }\ \omega \in \text{supp}(\mu) \ . \end{equation}

\noindent
Control systems involving unknown parameters have been well-studied in literature finding widespread applications particularly from the point of view of the robust (worst-case) control, see for instance the monographs \cite{ackermann2012robust}, \cite{warga2014optimal} and 
\cite{boltyanski2012robust} (and the references therein), and the paper \cite{vinter2005minimax} on minimax optimal control. 
In the introductory section of \cite[Chapter IX]{warga2014optimal}, control problems with uncertainties are considered comparing the conservative approach (minimax) with an alternative approach in which one might minimize, for instance, an ``expected value'' (which corresponds to the average cost problem studied in our paper). Then, in \cite[Chapters IX and X]{warga2014optimal} Warga investigates the so-called ``conflicting/adverse control problems'' providing necessary conditions for this broad class of problems which covers minimax problems (under some regularity assumptions), but which does not cover optimal control problems having the average cost criterion studied in our paper.
(See \cite{warga1991nonsmooth} for further developments on adverse control problems in the nonsmooth context; cf. the recent papers \cite{palladino2016necessary} on adverse control problems and \cite{karamzin2016minimax} on state-constrained minimax problems.)

\noindent
A growing interest has recently emerged in considering an `averaged' (or `expected' with respect to a given measure) approach, 
exploring various issues, directions and applications: see for instance a recent series of papers on aerospace systems \cite{ross2015lebesgue}, \cite{ross2015riemann}, \cite{caillau2017solving}, and the articles \cite{agrachev2016ensemble} and \cite{zuazua2014averaged} on averaged controllability (from different viewpoints); see also \cite{veliov2008optimal} for results on heterogeneous systems.

\noindent
Therefore, motivated not only by theoretical reasons but also by a recent growing interest in applications (such as aerospace engineering, see in particular \cite{ross2015lebesgue} and \cite{ross2015riemann}), in our paper we consider the `average cost' paradigm rather than the more `classical' criteria employed in the minimax/robust or adverse optimization framework.

\noindent

For the general (nonsmooth) case we derive necessary optimality conditions ensuring the existence of a {\it costate function} $p(.,.) : [0,T]\times \Omega \to \R^n$ which satisfies an averaged (on $\Omega$) maximality condition. Moreover, the costate arcs $p(.,\omega)$'s satisfy also the somewhat expected adjoint system and transversality condition, when $\omega$ belongs at least to a countable dense subset $\widehat \Omega$ of supp$(\mu)$. We show that these last two necessary conditions extend to the whole supp$(\mu)$ for free right end-point problems, if we impose (suitable) regularity assumptions on the dynamics and the cost function. We also prove that a further (non-trivial) case, in which the conditions of maximum principle extend to the whole supp$(\mu)$, is when the measure $\mu$ is purely atomic (not necessarily with finite support).


\vskip2ex
\noindent
This paper is organized as follows. 
We first study the simpler case in which the measure $\mu$ has a finite support (Section 2), which constitutes a discretization model for the general case of an arbitrary measure on a complete separable metric space (which is investigated successively). The main results are displayed in Section 3, and their proofs are given in Section 5. Section 4 is devoted to recall some fundamental theorems in measure theory and provide a limit-taking lemma which play a crucial role in our analysis. The approach that we suggest in our paper consists in approximating the measure $\mu$ by measures with finite support (convex combination of Dirac measures). 
Owing to Ekeland's variational principle, we construct a suitable family of auxiliary optimal control problems, the solutions of which approximate the reference problem (\ref{intprob}). Invoking the maximum principle (applicable in a more traditional version) for the approximating minimizers, we obtain properties which, taking the limit (in a suitable sense), allow us to derive the desired necessary conditions. 
The most difficult part in our proof is to show the maximality condition: this requires non-trivial consideration of multifunction representation and selection theorems. 
This part becomes simpler for the `purely atomic' case and the `smooth' case.\\
An important source of inspiration for the techniques here employed is represented by Vinter's paper \cite{vinter2005minimax} (which is devoted to minimax optimal control but, in fact, contains flexible and effective analytical tools that can be extended or adapted to our case). As one may expect, the necessary conditions that we obtain differ from those ones in the minimax context (in particular for the general nonsmooth case and the purely atomic case), for the nature of the minimization criterion is different.
For instance, for the general (nonsmooth) case the most evident difference with respect to the costate arcs characterization given in \cite{vinter2005minimax} is that (avoiding a formulation which might involve somewhat complicate sets) we show that the `expected' adjoint system and transversality conditions are satisfied by a family of costate arcs $p(.,\omega)$'s, at least when the parameter $\omega$ belongs to a countable dense set $\widehat \Omega\subset \text{supp}(\mu)$.
We highlight that an important feature of our paper is the unrestrictive nature of our assumptions: indeed, we allow not only nonsmooth data (on the dynamics, the cost function and the  averaged right end-point constraint), but we also provide results for unknown parameters belonging to a mere complete separable metric space $\Omega$. This aspect is particularly  relevant for applications (cf. \cite{ross2015lebesgue}) where $\Omega$ (and the support of the reference measure $\mu$) need not to be compact. 
Our techniques could be used to generalize the conditions in \cite{vinter2005minimax} and might provide some insights into dealing with adverse/conflicting control problems with non-compact parameter sets (in \cite{warga2014optimal} and \cite{warga1991nonsmooth} parameter sets are assumed to be compact.)


\vskip2ex
\noindent
\textbf{Notation}
Let $( \Omega,\rho_\Omega)$ be a metric space. Denote by $\mathcal B_\Omega$ the $\sigma$-algebra of Borel sets in $\Omega$. A probability measure $\mu$ on the measurable space $(\Omega, \mathcal B_\Omega)$ takes non-negative values, verifies the $\sigma$-additivity property and is such that $\mu(\Omega)=1$. The family of all probability measures on $(\Omega,\mathcal B_\Omega)$ is denoted by $\mathcal M(\Omega)$. Recall that a sequence $\{ \mu_i \}$ of measures in $\mathcal M(\Omega)$ is said to converge weakly$^*$ to a measure $\mu \in \mathcal{M}(\Omega)$ (in symbol $\mu_i \stackrel{*}\rightharpoonup \mu$), if $\int_\Omega h d \mu_i \rightarrow \int_\Omega h d \mu $ for every bounded continuous function $h$ on $\Omega$. The support of a measure $\mu$ defined on $\Omega$ is written \text{supp}($\mu$).
$\mathcal L$ denotes the Lebesgue subsets of $[0,T]$, while $\mathcal B^{m}$ are the Borel subsets of $\R^m$. $\mathcal L \times \mathcal B^{m}$ (respectively  $\mathcal L \times \mathcal B^{m} \times \mathcal B_\Omega$) is the product $\sigma-$algebra of $\mathcal{L}$ and $ \mathcal B^{m}$ (respectively $\mathcal{L}$, $ \mathcal B^{m}$ and $\mathcal B_\Omega$). 
The Euclidean norm is written $|.|$. We shall employ the following norm on $W^{1,1}([0,T];\R^{n})$:
$\|x(.)\|_{W^{1,1}} ~:= ~|x(0)| + \|\dot x(.)\|_{L^1(0,T)} $.
We write $\partial \varphi (x)$ the limiting subdifferential of the (possibly extended valued) function $\varphi : \R^n \rightarrow \R \cup \{ + \infty\}$ at $x \in \text{dom} \varphi.$ If $\varphi = \varphi(x,y)$, then $\partial_x \varphi (x,y)$ is the partial limiting subdifferential with respect to the variable $x$. $\B$ is the closed unit ball in Euclidean space. $N_C(x)$ is the limiting normal cone of a closed set $C$ at a point $x\in C$, and $N^1_C(x):=N_C(x)\cap \B$. (We refer the reader to \cite{aubin2009set}, \cite{clarke1990optimization}, \cite{clarke2013functional}, and \cite{vinter2010optimal} and the references therein for these nonsmooth analytical tools.)

\section{Average on measures with finite support}
We start considering the particular and simple case of optimal control problems of the form (\ref{intprob}), where the probability measure $\mu$ of the integral functional has a finite support: it is a convex combination of unit Dirac measures. This constitutes also a preliminary step to derive necessary conditions for the general case.

\noindent
The following assumptions will be needed throughout this section. For a given $W^{1,1}-$local minimizer $(\bar u,\{\bar x(.,\omega): \omega \in \Omega\})$ and for some $\delta >0$, we shall suppose:
\begin{enumerate}[label=(H\arabic*),ref=H\arabic*]
	\item \label{H1}
	\begin{enumerate}[label=(\roman*),ref=(H\arabic*)(\roman*)]
		\item \label{H1(i)} The function $f(.,x,.,\omega)$ is $\mathcal L \times \mathcal B^{m}$ measurable for each $(x,\omega) \in \R^n \times \Omega.$ 
		\item \label{H1(ii)} The multifunction $t \leadsto U(t)$ has nonempty values, and $\textrm{Gr }U(.) $ is a $\mathcal L \times \mathcal B^{m}$ measurable set.
	\end{enumerate}	 
	\item  \label{H2}  There exists a $\mathcal L \times \mathcal B^{m}$ measurable function $k_f: [0,T] \times \R^m \rightarrow \R$ such that $t \rightarrow k_f(t,\bar u(t))$ is integrable, and for each $\omega \in \Omega$,
	\[\left| f(t,x,u,\omega)-f(t,x',u,\omega)\right| \leq k_f(t,u)|x-x'|\] for all $x, x' \in \bar x(t,\omega)+\delta \B$, $u \in U(t)$, a.e. $t \in [0,T]$.
	\item  \label{H3} The function $g(.,\omega)$ is Lipschitz continuous on $ \bar x (T,\omega)+\delta \B$ for all $\omega \in \text{supp}(\mu)$.
\end{enumerate}

\begin{proposition} \label{finitenco}
	Let  $(\bar u,\{\bar x(.,\omega): \omega \in \Omega \})$ be a $W^{1,1}-$local minimizer for (\ref{intprob}). Assume that $\mu$ is a given probability measure with finite support and that for some $\delta>0$, hypotheses (\ref{H1})-(\ref{H3}) are satisfied. Then, there exist a family of arcs $\{p(., \omega) \in W^{1,1}([0,T], \R^n): \omega \in \Omega \}$ and a number $\lambda \ge 0$ such that
	\begin{enumerate}[label=(\alph*),ref=(\alph*)]
		\item \label{a_nco_finite_support} $(\lambda, p(.,\omega)) \neq (0,0)$ for all $\omega \in \Omega$ ;
		\item\label{b_nco_finite_support} $ \begin{aligned}[t] \int_\Omega  p(t,\omega) \cdot & f(t,\bar x(t,\omega) ,\bar u(t),\omega) \ d\mu(\omega)  = \max_{u \in U(t)} \int_\Omega p(t,\omega) \cdot f(t,\bar x(t,\omega),u,\omega)\ d\mu(\omega) \qquad \text{a.e. } t \in [0,T] \ ; \end{aligned}$
		\item \label{c_nco_finite_support} $- \dot p(t, \omega) \in \textrm{co } \partial_x [p(t,\omega) \cdot f(t,\bar x(t,\omega), \bar u(t), \omega)]$ \qquad for $ \mu-\textrm{a.e. } \omega \in \Omega$ ;
		\item \label{d_nco_finite_support}  $- p(T, \omega) \in \lambda \partial_x g (\bar x(T, \omega); \omega)+  N_{C(\omega)}(\bar x (T, \omega)) $\qquad for $ \mu-\textrm{a.e. } \omega \in \Omega$.
	\end{enumerate}
\end{proposition}

\begin{proof}
	
	The measure $\mu$ can be written as a convex combination of Dirac measures at points $\omega_j \in \Omega$, for $j=1,\ldots,N$, where $N$ is a suitable integer, as follows:
	\begin{equation} \label{definition_dirac_measure_convex_combination} \mu = \sum_{j=1}^{N} \alpha_j \delta_{\omega_j} \ , \qquad \sum_{j=1}^{N} \alpha_j =1\ , \quad \alpha_j \in (0,1] \ .     \end{equation}
	As a consequence the integral functional to minimize (\ref{intprob}) reduces to the following finite sum:
	\[ \sum_{j=1}^{N} \alpha_j g (x(T,\omega_j);\omega_j) \ ,  \]
	and, the minimization problem (\ref{intprob}) turns out to be easily treated, for it can be equivalently written as a standard optimal control problem: \leqnomode		
	\begin{equation*}\label{finiteprob} \begin{cases}
	\begin{aligned}
	& {\text{minimize}}
	& &  \sum_{j=1}^{N} \alpha_j g(x(T,\omega_j); \omega_j) \\
	&&& \hspace{-1.9cm}\text{over controls } u(.) \text{ such that } u(t) \in U(t) \textrm{ a.e. } t \in [0,T] \\ &&&  \hspace{-2.0cm} \text{ and arcs } x(.,\omega_j) \text{ such that for each } j=1,\ldots,N \\
	&&& \dot{x}(t,\omega_j) = f(t,x(t,\omega_j), u(t), \omega_j) \quad \textrm{a.e. } t \in [0,T] \\
	&&& x(0,\omega_j) = x_0 \quad \text{and} \quad x(T,\omega_j) \in C(\omega_j)  \  .  \end{aligned} \tag{$P_N$}
	\end{cases}	
	\end{equation*}
		Observe that in writing (\ref{finiteprob}), we can restrict attention only to elements $\omega$ belonging to the $\text{supp}(\mu) = \{ \omega_1,\ldots,\omega_N\}$. Under the stated assumptions (\ref{H1})-(\ref{H3}) and using the sum rule (cf. \cite[Theorem 5.4.1]{vinter2010optimal}), the necessary conditions for (\ref{finiteprob}) can be derived from the nonsmooth maximum principle \cite[Theorem 6.2.1]{vinter2010optimal} which guarantees the existence of a multiplier $\lambda \ge 0$ and arcs $\widetilde p(., \omega _j) \in W^{1,1}([0,T],\R^n)$, for $j=1,\ldots,N$ such that
	\begin{enumerate}[label= (\roman{*}), ref= \roman*]
		\item $ (\lambda,  \widetilde p(.,\omega_1),\ldots, \widetilde p(.,\omega_N)) \neq (0,\ldots,0)$ ;
		\item $- \dot {\widetilde p}(t, \omega_j) \in \textrm{co } \partial_x [\widetilde p(t,\omega_j) \cdot f(t,\bar x(t,\omega_j), \bar u(t), \omega_j)]$ \quad for all $j=1,\ldots,N$ ;
		\item $ - \widetilde p(T, \omega_j) \in \lambda \alpha_j \partial_x g (\bar x (T,\omega_j); \omega_j) +   N_{C(\omega_j)}( x (T, \omega_j)) $\quad for all $j=1,\ldots,N$ ;
		\item $\sum \limits_{j=1}^{N} \widetilde p(t, \omega_j) \cdot f (t, \bar x (t, \omega_j), \bar u(t), \omega_j) = \max \limits_{u \in U(t)}  \sum \limits_{j=1}^{N} \widetilde p(t, \omega_j) \cdot f (t, \bar x (t, \omega_j), u, \omega_j) \quad$ a.e. $t\in [0,T]$.
	\end{enumerate}
	For each $j$, we set \[ p (.,\omega_j) := \frac{\widetilde p(.,\omega_j)}{\alpha_j} \ . \] 
We deduce, therefore, conditions \ref{a_nco_finite_support}-\ref{d_nco_finite_support} of the proposition statement. This concludes the proof.

\end{proof}

\section{Main results}

We take now a probability space $(\Omega, \mathcal B_\Omega,\mu)$ where $\mu$ is a (general) probability measure. 
For a given  $W^{1,1}-$local minimizer $(\bar u,\{\bar x(.,\omega): \omega \in \Omega \})$ and for some $\delta >0$, we shall suppose: 
\begin{enumerate}[label=(A\arabic*), ref=A\arabic*]
	\item \label{A3_nco_general_case} $(\Omega, \rho_\Omega)$ is a complete separable metric space. 
	\item \label{A1_nco_general_case}
	\begin{enumerate}[label=(\roman*), ref=(A2)(\roman{*})]
		\item \label{A1_i_nco_general_case} The function $f(.,x,.,.)$ is $\mathcal L \times \mathcal B^m \times \mathcal B_\Omega$ measurable for each $x \in \R^n$.
		\item \label{A1_ii_nco_general_case} The multifunction $t \leadsto U(t)$ has nonempty values and $\textrm{Gr }U(.) $ is a $\mathcal{L} \times \mathcal{B}^m$ measurable set.
		\item \label{A1_iii_nco_general_case} The set $ f(t,x,U(t),\omega)$ is closed for all $x \in \bar x(t,\omega) + \delta\B$, and $(t,\omega) \in [0,T] \times \Omega.$
	\end{enumerate}	 
	\item \label{A2_nco_general_case} There exist a constant $c >0$ and an integrable function $k_f: [0,T] \rightarrow \R$ such that
			\[\left| f(t,x,u,\omega)-f(t,x',u,\omega)\right| \leq k_f(t)|x-x'| \;\; \textrm{and} \;\; |f(t,x,u,\omega)| \leq c \] for all $x, x' \in \bar x(t,\omega) + \delta \B$, $u \in U(t)$, $\omega \in \Omega$ a.e. $t \in [0,T]$.
	\item \label{A4_nco_general_case}
	\begin{enumerate}[label=(\roman{*}),ref=(A4)(\roman{*})]
		\item \label{A4_i_nco_general_case} The function $g$ is $\mathcal B^n \times \mathcal{B}_\Omega$ measurable.
		\item \label{A4_ii_nco_general_case}There exist positive constants $k_g \ge 1$ and $M \ge \delta$ such that for all $\omega \in \Omega$ we have \vskip1ex
		 $|g(x,\omega)| \leq M$ and $d_{C(\omega)}(x) \le M$ for all $x \in \bar x(T,\omega) + \delta \B$
		,\vskip1ex
		 $ |g (x, \omega) - g(x', \omega)| \leq k_g |x-x'|$  for all $x, x' \in \bar x(T,\omega) + \delta \B \ .$ 
		\item \label{A4_iii_nco_general_case}
		There exists a modulus of continuity $\theta(.)$ such that for all $\omega \in \Omega$ and $x \in \bar x(T,\omega) + \delta \B$ we have
		$$
		|g(x,\omega_1)-g(x,\omega_2)| \leq \theta(\rho_\Omega (\omega_1,\omega_2)) \quad \mbox{for all } \; \omega_1,\omega_2 \in \Omega\ , 
		$$
		and
		$$
		|d_{C(\omega_1)}(x)-d_{C(\omega_2)}(x)| \leq \theta(\rho_\Omega (\omega_1,\omega_2)) \quad \mbox{for all } \; \omega_1,\omega_2 \in \Omega\ .
		$$
	\end{enumerate} 
	\item \label{A5_nco_general_case}
	
	 There exists a modulus of continuity $\theta_f(.)$ such that for all $\omega,\omega_1, \omega_2 \in \Omega$, \[ \int_{0}^{T} \sup\limits_{x \in \bar x (t,\omega) + \delta \B, \; u \in U(t)}
		|f(t,x,u,\omega_1) - f(t,x,u,\omega_2) | \ d t \leq \theta_f (\rho_{\Omega}(\omega_1, \omega_2)). \]
\end{enumerate}
(We say that $\theta: [0, \infty) \rightarrow [0,\infty)$ is a modulus of continuity if $\theta(s)$ is increasing and $\lim\limits_{ s \downarrow 0} \theta (s)=0.$)

\noindent
The first result provides necessary optimality conditions for the general nonsmooth case.

\begin{theorem} \label{theorem_nco_general_case_probability_space}
	Let  $(\bar u,\{\bar x(.,\omega): \omega \in \Omega \})$ be a $W^{1,1}-$local minimizer for (\ref{intprob}) in which $\mu \in \mathcal{M}(\Omega)$ is given. Assume that, for some $\delta>0$, hypotheses (\ref{A3_nco_general_case})-(\ref{A5_nco_general_case}) are satisfied. Then, there exist $\lambda \ge 0$, a $\mathcal{L}\times\mathcal{B}_\Omega$ measurable function $p(.,.): [0,T] \times \Omega \rightarrow \R^n$ and a countable dense subset $\widehat \Omega$ of supp$(\mu)$ such that
	\begin{enumerate}[label=(\roman*), ref=\roman*]
		\item \label{item: absolute continuity of adjoint arcs general case} $p(.,\omega) \in W^{1,1}([0,T],\R^n)$ \quad for all $\omega \in \widehat \Omega \ ;$
		\item \label{item: maximality condition general case} $\begin{aligned}[t]  \int_\Omega p(t,\omega) \cdot  f(t,\bar x(t,\omega),\bar u(t),\omega) \ d\mu(\omega)  = \max_{u \in U(t)} \int_\Omega p(t,\omega) \cdot f(t,\bar x(t,\omega),u,\omega) \ d\mu(\omega) \quad \text{a.e. } t \in [0,T] \ ; \end{aligned} $
		\item \label{item: adjoint system+nontriviality+transversality: set L(omega)}$p(.,\omega) \in \text{co }\mathcal{P}(\omega)$ \quad for all $\omega \in \widehat{\Omega}$
		where		
	\begin{align} \nonumber \mathcal{P}(\omega) & := \Bigg\{ q(.,\omega) \in  W^{1,1}( [0,T], \R^n) \ : \  
	\label{non triviality condition in the set L(omega)} ( \lambda, \{q(.,\omega) : \omega \in \widehat{\Omega} \}) \neq (0,0) ,
\\ & \nonumber - \dot q(t,\omega) \in \text{co }\partial_x [q(t,\omega) \cdot f(t,\bar x(t,\omega),\bar u(t),\omega)] \quad \text{ a.e. } t \in [0,T],\\ & \text{and }\
-q(T,\omega) \in \lambda \partial_x g(\bar x(T,\omega);\omega) +  N^1_{C(\omega)}(\bar x(T,\omega))  \Bigg\}. \nonumber \end{align}
	\end{enumerate}

\end{theorem}

\noindent
Moreover, we consider two special cases in which condition (iii) becomes much simpler and the desired properties involving the costate arcs extend to the whole $\mbox{supp} (\mu)$: when the measure $\mu$ is purely atomic, and the smooth right end-point free case. 

\begin{theorem}[Purely atomic case]\ \label{atomic}
	Let  $(\bar u,\{\bar x(.,\omega): \omega \in \Omega \})$ be a $W^{1,1}-$local minimizer for (\ref{intprob}) in which $\mu \in \mathcal{M}(\Omega)$ is a purely atomic measure such that each atom is a singleton. Assume that, for some $\delta>0$, hypotheses (\ref{A3_nco_general_case})-(\ref{A5_nco_general_case}) 
	are satisfied. Then, there exist $\lambda \ge 0$, a $\mathcal{L}\times\mathcal{B}_\Omega$ measurable function $p(.,.): [0,T] \times \Omega \rightarrow \R^n$ and a (at most) countable set $\widehat \Omega=$supp$(\mu)$ such that
	\begin{enumerate}[label=(\roman*), ref=\roman*]
		\item \label{item: absolute continuity of adjoint arcs atomic case} $p(.,\omega) \in W^{1,1}([0,T],\R^n)$ \quad for all $\omega \in \widehat \Omega \ ;$
		\item \label{item: maximality condition atomic case} $\begin{aligned}[t]  \int_\Omega p(t,\omega) \cdot  f(t,\bar x(t,\omega),\bar u(t),\omega) \ d\mu(\omega)  = \max_{u \in U(t)} \int_\Omega p(t,\omega) \cdot f(t,\bar x(t,\omega),u,\omega) \ d\mu(\omega) \quad \text{a.e. } t \in [0,T] \ ; \end{aligned} $
		\item \label{item: adjoint system+nontriviality+transversality atomic} 
		$( \lambda, \{p(.,\omega) : \omega \in \widehat{\Omega} \}) \neq (0,0)$, and for all $\omega \in \widehat{\Omega}=$supp$(\mu)$
		\begin{align} \nonumber &   
		\\ & \nonumber - \dot p(t,\omega) \in \text{co }\partial_x [p(t,\omega) \cdot f(t,\bar x(t,\omega),\bar u(t),\omega)] \quad \text{ a.e. } t \in [0,T],\\ & \text{and }\
		-p(T,\omega) \in \lambda \partial_x g(\bar x(T,\omega);\omega) + N^1_{C(\omega)}(\bar x(T,\omega))  \ . \nonumber
		\end{align}
	\end{enumerate}

\end{theorem}


\begin{theorem}[Smooth case]\label{proposition: smooth data nco}
	Let  $(\bar u,\{\bar x(.,\omega): \omega \in \Omega \})$ be a $W^{1,1}-$local minimizer for (\ref{intprob}) where $\mu \in \mathcal{M}(\Omega)$ is given. Suppose that, for some $\delta>0$, hypotheses (\ref{A3_nco_general_case})-(\ref{A2_nco_general_case}), (\ref{A4_nco_general_case})(i) and (\ref{A5_nco_general_case}) are satisfied. In addition, assume that
	\begin{enumerate}[label=(C\arabic*), ref=C\arabic*]
		\item \label{item: c1 assumption smooth data proposition nco} $g(.,\omega)$ is differentiable on $\bar x(T,\omega) + \delta \B$, for each $\omega \in \Omega$, and $\nabla_x g(.,.)$ is continuous;
		
		
		\item\label{item: c2 assumption smooth data proposition nco} $f(t,.,u,\omega)$ is continuously differentiable on $\bar{x}(t,\omega)+\delta \B$ for all $u\in U(t)$ and $\omega \in \Omega$ a.e. $t\in [0,T]$, and $\omega \rightarrow \nabla _x f(t,x,u,\omega)$ is uniformly continuous with respect to $(t,x,u) \in \{ (t',x',u')  \in [0,T] \times \R^n \times \R^m \ | \ u' \in U(t')  \}.$
		
		\item $C(\omega) := \R^n$.
	
	\end{enumerate}
	
 Then, there exists a $\mathcal{L}\times\mathcal{B}_\Omega$ measurable function $p(.,.): [0,T] \times \Omega \rightarrow \R^n$ such that	
	\begin{enumerate}[label=(\roman*)$'$, ref=(\roman*)$'$]
		\item \label{item: absolute continuity of adjoint arcs corollary} $p(.,\omega) \in W^{1,1}([0,T],\R^n)$ \quad for all $\omega \in $supp$(\mu)$; 
		\item \label{item: maximality condition corollary} $\begin{aligned}[t]  \int_\Omega p(t,\omega) \cdot  f(t,\bar x(t,\omega),\bar u(t),\omega) \ d\mu(\omega)  = \max_{u \in U(t)} \int_\Omega p(t,\omega) \cdot f(t,\bar x(t,\omega),u,\omega) \ \textrm{d}\mu(\omega) \quad \text{a.e. } t \in [0,T] \ ; \end{aligned} $
		\item \label{item: adjoint system corollary} 
		$-\dot p(t, \omega) = [\nabla_x f(t, \bar x (t, \omega), \bar u(t), \omega)]^T   p(t, \omega) $\quad a.e. $t\in [0,T]$, \ for all $\omega \in$ supp$(\mu)$; 
		\item \label{item: transversality condition corollary}
		$- p(T, \omega)  = \nabla_x g (\bar{x}(T,\omega);\omega), $\quad for all $\omega \in $supp$(\mu)$. 
\end{enumerate}	
	
\end{theorem}

\vskip3ex
\noindent
{\bf Comments}
\vskip2ex

	\noindent
	Condition (\ref{item: adjoint system+nontriviality+transversality: set L(omega)}) of Theorem \ref{theorem_nco_general_case_probability_space} is interpreted in the following sense: for each $\omega \in \widehat{\Omega}$, one considers functions $q(.,\omega) \in W^{1,1}([0,T], \R^n)$ (such that $ \| q(.,.) \|_{L^\infty}$ is uniformly bounded by a constant) satisfying the adjoint system
	\[   - \dot q(t,\omega) \in \text{co }\partial_x [q(t,\omega) \cdot f(t,\bar x(t,\omega),\bar u(t),\omega)] \quad \text{ a.e. } t \in [0,T] \ , \]
	and the transversality condition
	\[-q(T,\omega) \in \lambda \partial_x g(\bar x(T,\omega);\omega) + N^1_{C(\omega)}(\bar x(T,\omega))  \ . \]
	Then, from this set of functions, one takes into account only the $q(.,.)$'s such that
	\[ ( \lambda, \{q(.,\omega) : \omega \in \widehat{\Omega} \}) \neq (0,0) \]
	to generate the family of arcs sets of $\{ \mathcal{P}(\omega)\}_{\omega \in \widehat{\Omega}}$.

\medskip

\noindent			 
In optimal control theory, necessary optimality conditions results are usually provided avoiding the `trivial' case, which is given by the couple $(\lambda, p(.,.))=(0,0)$, where $\lambda$ is the multiplier associated with the cost. 
However, in literature dealing with optimal control problems with unknown parameters in the non-smooth context, results are often written including possible trivial cases which are not considered so relevant for the general properties expressed in the results statement; cf. \cite{vinter2005minimax} on nonsmooth minimax problems and \cite{warga1991nonsmooth} on nonsmooth adverse problems, in which the operator `${\rm co}$' (convexifying over sets of costate arcs) is considered possibly bringing trivial cases. (The fact that in \cite{vinter2005minimax} and \cite{warga1991nonsmooth} the multiplier associated with the cost $\lambda$ does not appear in the necessary conditions should not be so surprising: this multiplier is somewhat hidden in the analysis and, in these contexts, the situation `$p\equiv 0$' alone might be considered as `trivial'). 
In our case, we might have a trivial couple $(\lambda=0, p(.,.)=0)$ which satisfies the conditions of Theorem \ref{theorem_nco_general_case_probability_space}, indeed, employing the convexification operator `${\rm co}$' on the set of costate arcs, it may happen that, taking $\lambda=0$, even if $ p(.,\hat \omega)\neq 0$, with $\hat \omega \in \widehat \Omega$, also $-p(.,\hat \omega)$ is an admissible costate arc; convexifying, $p\equiv 0 \in \text{co }\mathcal{P}(\hat \omega)$.  
We decided to be consistent with part of previous (nonsmooth) literature on problems with unknown parameters and provide a general nonsmooth result (Theorem \ref{theorem_nco_general_case_probability_space}), which allows (in some particular circumstances) a trivial case, but at the same time covers a number of non-restrictive non-trivial cases. For instance,  (\ref{item: adjoint system+nontriviality+transversality: set L(omega)}) of Theorem \ref{theorem_nco_general_case_probability_space} immediately implies a non-triviality condition for the pair $(\lambda,p(.,.))$ when
			\begin{enumerate}
				\item[(a)]	the right end-point constraints are absent ($C(\omega)\equiv \R^n$);
				\item[(b)]	the given measure $\mu$ has a nonatomic component, the averaged right end-point constraints \[ \int_{\Omega} d_{C(\omega)} (x(T,\omega)) \ d\mu(\omega)=0\] are imposed  but the normal cone to the end-point constraint ${\rm co } N_{C(\omega)}(\bar x (T, \omega))$ is pointed for all $\omega \in \Omega$ (or even for $\omega$ belonging to a countable dense subset of the support of the nonatomic component of $\mu$). We recall that a convex cone $K \subset \R^n$ is said to be `pointed' if for any nonzero elements $d_1, \ d_2 \in K,$ $d_1 + d_2 \neq 0$.
			\end{enumerate}
		Concerning (b), the abnormal situation (i.e. $\lambda=0$) is admissible, but the fact that ${\rm co } N_{C(\omega)}(\bar x (T, \omega))$ is pointed ensures that $p\equiv 0 \notin \text{co }\mathcal{P}(\hat \omega)$ for all $\hat \omega \in \widehat{\Omega}$.
			
			\medskip
			
\noindent
			The `degeneracy issue' (i.e. the necessary conditions are satisfied by any control) is a longstanding issue which has been widely investigated in optimal control. It is well-known that this issue may arise, for instance, in presence of state constraints for `standard' (in the sense that parameters are absent) optimal control problems (cf. \cite[Chapter X]{vinter2010optimal} and the references therein). Rather less is known for the case when unknown parameters intervene in the dynamics and the cost: minimax, adverse, and average optimal control problems. (See \cite{karamzin2016minimax} for a non-degeneracy result on state constrained minimax problems avoiding the degeneracy caused by the state constraint; see also \cite{vinter2005minimax} for a link between minimax and state-constrained problems). In our context degeneracy might occur for the general nonsmooth case (Theorem \ref{theorem_nco_general_case_probability_space}) when the given measure $\mu$ has a nonatomic component.
			Indeed, our construction of the costate arcs $p(t,\omega)$ for $\omega\in \Omega$ is based on a limit-taking procedure starting from the information provided by (non-trivial) costate arcs $p(t,\hat \omega)$ for $\hat \omega\in \widehat{\Omega}$ (cf. (\ref{limit}) below).  If $\mu$ has a nonatomic component, we have no reason to expect (under the general assumptions considered in Theorem \ref{theorem_nco_general_case_probability_space}) that the non-degenerate property of the costate arcs $p(t,\hat \omega)$ ($\hat \omega\in \widehat{\Omega}$) always propagates on $\Omega$ as desired: there might be some degenerate situations in which for a full-measure subset of $\Omega$ the limit we take in the proof of Theorem \ref{theorem_nco_general_case_probability_space} does not exist, and $p(.,.)$ extends with the value zero on $\Omega\setminus \widehat{\Omega}$, obtaining a degeneracy issue. 
			However, under some circumstances, the information provided on the set $\widehat{\Omega}$ does propagate: if there is no right end-point constraint and, in addition, we impose regularity assumptions on the dynamics and the terminal cost function, properties (\ref{item: absolute continuity of adjoint arcs general case}) and (\ref{item: adjoint system+nontriviality+transversality: set L(omega)}) of Theorem \ref{theorem_nco_general_case_probability_space} extend to the whole parameter set $\Omega$, as stated in Theorem \ref{proposition: smooth data nco}. 
			Theorems \ref{atomic} and \ref{proposition: smooth data nco} do provide non-degenerate results.
			%
			
	
	\vskip2ex
			\noindent	
			Nonsmooth results on optimal control problems with unknown parameters, such as adverse and minimax problems (see \cite{warga1991nonsmooth} and \cite{vinter2005minimax}), are concerned with a `degenerate issue' which is not far from the one of our nonsmooth result Theorem \ref{theorem_nco_general_case_probability_space}, maybe, in a more `dramatic' way, for the measure -appearing there as a multiplier in the necessary conditions- is not uniquely determined, and may have a support with degenerate effects on the necessary conditions.
Consider for instance the simple example \cite[Example 4.1]{vinter2005minimax} in the context of minimax problem:
			$$
			\left\{
			\begin{array}{lll}
			{\bf minimize} \; \max_{\omega \in \Omega:=[-1,1]} -|x(1)- \omega|\\
			\dot{x}(t) = u(t) \quad \textrm{a.e. } t \in [0,1]\\
			u(t) \in [-1,1] \quad \mbox{a.e.} \; t \in [0,1]\\
			\mbox{and } x(0)=0 .
			\end{array}
			\right.
			$$
			A minimax minimizer is: $(\bar{x}\equiv 0, \bar{u}\equiv 0)$. In \cite{vinter2005minimax} there is a detailed discussion comparing 
			\cite[Proposition 2.1]{vinter2005minimax} (finite parameter sets case) and \cite[Theorem 3.1]{vinter2005minimax} (general nonsmooth case), and the necessity of convexifying the set of costate arcs in the general nonsmooth case, for, otherwise, the necessary conditions would not be valid. In particular, in \cite{vinter2005minimax} the (Dirac) measure $\delta_{\omega=0}$ concentrated at $\omega=0$ (point at which the reference minimizer attains its maximum) is considered, for which ``an arbitrary collection of $W^{1,1}$ functions such that $p(.;\omega=0)\equiv 0$'' satisfies the necessary conditions of \cite[Theorem 3.1]{vinter2005minimax}. The counterpart of this choice is that it is degenerate: any control satisfies the necessary conditions of \cite[Theorem 3.1]{vinter2005minimax}. 
			
			\noindent
			One might go a little bit further in this direction, observing that degeneracy is -in fact- much more dramatic for this particular example: indeed, for any probability measure $\mu$ on the parameter set $\Omega=[-1,1]$ the maximality conditions of \cite[Theorem 3.1]{vinter2005minimax} are necessarily degenerate for the reference minimax minimizer $(\bar{x}\equiv 0, \bar{u}\equiv 0)$ (and the trivial case $p\equiv 0$ is also admitted). 
			On the other hand, if one is interested in the different performance criterion given by the average cost $\int_{\omega \in [-1,1]} -|x(1)- \omega| d\omega/2$ with the same dynamics, these dramatic issues of triviality and degeneracy disappear. (To see this, we can take, for instance, the average minimizer $(\bar{x}(t)=t, \bar{u}\equiv 1)$ associated with the costate $p\equiv 1$.)

			\medskip

\noindent
			At first glance our results might look similar to some statements on necessary conditions appearing in \cite{warga2014optimal} and \cite{warga1991nonsmooth}. Not only these results do not cover the class of average cost minimization problems (in the sense of our paper), but we also highlight a crucial aspect concerning the completely different role of the measures entering in the picture of the necessary conditions: 
in Warga's framework the existence of a positive Radon measure (on the set of ‘adverse’ relaxed controls) is a necessary condition, and this measure plays the role of a `multiplier'. In our context (of average control problems) the probability measure $\mu$ is a given datum, and we underline the fact that our objective is to give necessary conditions w.r.t. the given measure $\mu$.
			
			\medskip
\noindent
			We finally observe that the construction of the countable set $\widehat \Omega$ proposed in this paper could be useful for applications: it provides a constructive way to approximate the reference measure $\mu$ by means of a sequence of convex combinations of Dirac measures concentrated at points of $\widehat \Omega$.  Therefore the set $\widehat \Omega$ can be considered as a reference set of parameters $\omega$'s for which one starts computing the costate arcs $p(.,\omega)$ and, eventually, derives conditions for optimal controls. 

\section{Preliminary results in measure theory}

This section is devoted to display results which will be relevant for the proofs of Theorems \ref{theorem_nco_general_case_probability_space}, \ref{atomic} and \ref{proposition: smooth data nco}. 
We shall make repeatedly use of the following theorem (also referred to as Portmanteau Theorem, cf. \cite[Theorem 4.5.1]{ash2014measure} or \cite[Theorem 6.1. pp. 40]{parthasarathy2005probability}) which provides conditions  characterizing the weak$^*$ convergence of probability measures on a metric space $(\Omega,\rho_\Omega)$. 

\begin{theorem} \label{portmanteauTheorem} Let $(\Omega,\rho_\Omega)$ be a metric space. Take a sequence of measures $\{ \mu_i  \}$ in $ \mathcal{M}(\Omega)$ and a measure $\mu \in \mathcal{M}(\Omega)$. The following conditions are equivalent:
	\begin{enumerate}[label=(\alph*), ref=\alph*]
		\item \label{(a0)_portmanteau} $\int_\Omega h \textrm{d} \mu_i \rightarrow \int_\Omega h \textrm{d} \mu $ for any bounded continuous function $h$ on $\Omega$ (i.e. $ \mu_i \stackrel{*}\rightharpoonup \mu $)\ ;
		\item \label{(a)_portmanteau} $\int_\Omega h \textrm{d} \mu_i \rightarrow \int_\Omega h \textrm{d} \mu $ for any bounded uniformly continuous function $h$ on $\Omega$\ ;
		\item \label{(b)_portmanteau} $\lim \mu_i (B) = \mu(B)$ for every Borel set $B$ whose boundary has $\mu-$measure zero. (Such sets are also referred to as $\mu-$continuity sets)\ ;
		\item \label{(c)_portmanteau} $\limsup \mu_i (C) \leq \mu(C)$ for every closed set $C$ in $\Omega$\ ;
		\item \label{(d)_portmanteau}$\liminf \mu_i (E) \geq \mu(E)$ for every open set $E$ in $\Omega$\ . 
	\end{enumerate}
\end{theorem}
\noindent

\noindent
We recall that $\mu \in \mathcal{M}(\Omega)$ is said to be {\it tight} if for each $\varepsilon>0$, there exists a compact set $K_\varepsilon \subset \Omega$ such that $\mu(\Omega \setminus K_\varepsilon) < \varepsilon$. A very well-known result asserts that when $(\Omega,\rho_\Omega)$ is a complete separable metric space, then every $\mu \in \mathcal{M}(\Omega)$ is tight (cf. \cite[Theorem 3.2.  pp. 29]{parthasarathy2005probability}).
We shall invoke also a generalized version of the Prokhorov's Theorem \cite[Theorem 8.6.2]{bogachev2007measure} which provides a useful characterization of the relatively compact subsets of Borel measures on $\Omega$, when $\Omega$ is a complete separable metric space. 
This result will be crucial to derive measure convergence properties (see Lemma \ref{weakconvergence} below).

\begin{theorem}[Generalized Prokhorov Theorem] \label{theorem_prokhorov_tight_relativelyCompact} \label{theorem_prokhorov_tight_relativelyCompact_generalized} Let $(\Omega,\rho_\Omega)$ be a complete separable metric space and consider a family $\Upsilon$ of Borel measures on $\Omega$. Then, $\Upsilon$ is relatively compact if and only if $\Upsilon$ is uniformly tight and uniformly bounded in the variation norm; in particular a sequence of measures $\{\mu_i\}$ admits a weakly$^*$ convergent subsequence if and only if the sequence $\{\mu_i\}$ is uniformly tight and uniformly bounded in the variation norm.
\end{theorem}

\noindent
We consider now subsets $D$ and $D_i$, for $i=1,2, \ldots$, of $\Omega \times \R^K$. We denote respectively by $D(.), \ D_i(.): \Omega \leadsto \R^K$ the multifunctions defined as
\[ D(\omega) := \{ z \in \R^K \ : \ (\omega,z) \in D    \} \quad \text{and }   \quad
 D_i(\omega) := \{ z \in \R^K \ : \ (\omega,z) \in D_i    \} \quad \text{for all } i=1,2,\ldots \ .  \]
Let $\{\mu_i\}$ be a weak$^*$ convergent sequence of measures in $\mathcal{M}(\Omega)$. Our aim is to justify the limit-taking of sequences like
\[ d\eta_i(\omega) = \gamma_i(\omega)\  d \mu_i(\omega) \qquad i=1,2,\ldots   \] in which $\{ \gamma_i (\omega) \}$ is a sequence of 
Borel measurable functions satisfying \[\gamma_i (\omega) \in D_i(\omega) \quad \mu_i- \textrm{a.e. } \]

\noindent
The required convergence result is provided by Lemma \ref{weakconvergence} below, which represents an extension of \cite[Proposition 9.2.1]{vinter2010optimal} and \cite[Proposition 6.1]{vinter2005minimax} to the case in which $\Omega$ is an arbitrary complete separable metric space (not necessarily compact).

\begin{lemma}
	 \label{weakconvergence} Let $\Omega$ be a complete separable metric space. Consider a sequence of measures $\{\mu_i\}$ in $\mathcal{M}(\Omega)$ such that $\mu_i \stackrel{*}\rightharpoonup \mu$ for some $\mu \in \mathcal{M}(\Omega)$, a sequence of sets $\{ D_i \subset \Omega \times \R^K \}$ such that
	 \begin{equation} \label{limsup_inclusion} \limsup\limits_{i \rightarrow \infty} D_i \subset D \ , \end{equation}
	 for some closed set $D \subset \Omega \times \R^K$, and a sequence $\{ \gamma_i : \Omega \rightarrow \R^K \}$ of Borel functions.
	 Suppose that
	 \begin{enumerate}[label = (\roman{*}), ref= \roman{*}]
	 	\item \label{item: i of proposition weak convergence} $D(\omega)$ is convex for each $\omega \in \text{dom }D(.)$;
	 	\item\label{item: ii of proposition weak convergence} the multifunctions $\omega \leadsto D(\omega)$ and $\omega \leadsto D_i(\omega)$, for all $i$, are uniformly bounded;
	 	\item\label{item: iii of proposition weak convergence} for each $i=1,2,\ldots$, 
		$ \gamma_i(\omega) \in D_i(\omega) \quad \mu_i-\text{a.e.}  \text{ and }  \text{supp}(\mu_i) \subset \text{dom }D_i(.) .$
	 \end{enumerate}
	 Define, for each $i$, the vector of signed measures $\eta_i:=\gamma_i\mu_i$. Then, along a subsequence, we have
	 \[  \eta_i \stackrel{*}\rightharpoonup \eta    \] where $\eta$ is a vector-valued Borel measure on $\Omega$ such that
	 \[ d\eta (\omega) = \gamma(\omega) \ d\mu(\omega)    \] for some Borel measurable function $\gamma : \Omega \rightarrow \R^K$ satisfying
	 \[ \gamma(\omega) \in D(\omega)   \quad \mu-\text{a.e. } \omega \in \Omega  \ . \]
	\end{lemma}
	\noindent
(The upper limit in (\ref{limsup_inclusion}) above must be understood in the Kuratowski sense, cf. \cite{aubin2009set} or \cite{vinter2010optimal}.)
 
 	\begin{proof} Since $\Omega$ is a complete separable metric space, the sequence $\{ \mu_i\}$ turns out to be uniformly tight as result of Theorem \ref{theorem_prokhorov_tight_relativelyCompact}. We also know that $\gamma_i(\omega) \in D_i(\omega) \ \mu_i-$ a.e. and $D_i(\omega)$ is uniformly bounded for all $i$. It follows that there exists a constant $M>0$ such that
 		\begin{equation}  \label{condition_gamma_bounded} |\gamma_i(\omega)| \leq M \quad \mu_i- \textrm{ a.e. }   \end{equation}
 		For each $i$, the vector-valued measure $\eta_i = \gamma_i \mu_i$ can be expressed as $\eta_i = (\eta_{i,1},\ldots,\eta_{i,K})$. From the tightness of $\{ \mu_i \}$ and (\ref{condition_gamma_bounded}), it immediately follows that, for all $k\in \{1,\ldots,K \}$, $\{ \eta_{i,k}  \}$ is a family of uniformly tight, possibly signed measures. Therefore according to Theorem \ref{theorem_prokhorov_tight_relativelyCompact}, for each $k\in \{1,\ldots,K \}$ one can extract a subsequence  $\{ \eta_{i,k} \}$ (we do not relabel) which converges weakly$^*$ to some $\eta_k$.
 		We show that $\eta:=(\eta_1,\ldots,\eta_K)$ is absolutely continuous with respect to $\mu$. Let $\eta_{i,k}= \eta_{i,k}^+ - \eta_{i,k}^-$ and $\eta_k= \eta_k^+ - \eta_k^-$ be the Jordan decompositions of $\eta_{i,k}$ and $\eta_k$, where $\eta_k^+$ and $\eta_k^-$ are respectively the weak$^*$ limits of $\eta_{i,k}^+$ and $\eta_{i,k}^-$. Let $B_{\eta,\mu}$ be the common family of continuity sets (in the sense of (\ref{(b)_portmanteau}) of Theorem \ref{portmanteauTheorem}) for the measures $\eta_1^+, \ldots, \eta_K^+$, $\eta_1^-,\ldots,\eta_K^-$ and $\mu$. Take any Borel set $B$ in $B_{\eta,\mu}$, we have
 		\begin{align*}
 		\left| \int_B d \eta   \right|  = \lim\limits_{i} \left| \int_B d \eta_i \right| = \lim\limits_{i} \left| \int_B \gamma_i(\omega) d \mu_i (\omega)  \right|  \leq M \lim\limits_{i} \int_B d \mu_i (\omega) = M \int_B d \mu (\omega) \ .
 		\end{align*}
 		But since $B_{\eta,\mu}$ generates all the Borel sets of $\Omega$ (cf. \cite[Chapter 7, Appendix]{khalil2017optimality}),
 		 it follows that $\eta$ is absolutely continuous with respect to $\mu$. Therefore, by the Radon-Nikodym Theorem, there exists a $\R^K$-valued, Borel measurable and $\mu$-integrable function $\gamma$ on $\Omega$ such that for any Borel subset $B$ of $\Omega$ we have
 		\[ \eta(B) = \int_B d \eta(\omega) = \int_B \gamma(\omega) d \mu (\omega) \ ; \]
 		equivalently,
 		\[ d \eta (\omega) = \gamma(\omega)  d \mu(\omega) \ .  \]
 		It remains to show that $\gamma(\omega) \in D(\omega) \; \mu-$a.e. $\omega \in \Omega.$ For all $j \in \N$ fixed, following the approach suggested in \cite[Proposition 9.2.1]{vinter2010optimal}, we define $D^j(\omega) := D(\omega)+\frac{1}{j}\B \subset \R^K$. We fix $q \in \R^K$. Since $D(\omega)$ is uniformly bounded and $D$ is closed, the multifunction $D^j(.)$ is upper semicontinuous. Then, for $\bar R >0$ large enough, the marginal function defined by 	
 		\[  \sigma_q(\omega) := \begin{cases} \max \{ q \cdot d \ : \ d \in D^j(\omega) \} \quad & \text{if } D^j(\omega) \neq \emptyset \\ \bar R & \text{otherwise}
 		\end{cases}  \] turns out to be upper semicontinuous and bounded on $\Omega$, owing to the Maximum Theorem (cf. \cite[Theorem 1.4.16]{aubin2009set}). From standard results on semicontinuous maps (cf. \cite[A6.6]{ash2014measure}), there exists a sequence of bounded continuous functions $\{ \psi_q^\ell: \Omega \rightarrow \R   \ , \ \ell=1,2,\ldots\}$ such that:
 		 \begin{equation}\label{lowerSemicontinuous_continuous} \lim\limits_{\ell \rightarrow \infty} \psi_q^\ell(\omega) = \sigma_q(\omega)  \quad \text{and} \quad  \sigma_q(\omega) \leq \psi_q^\ell (\omega) \quad \textrm{for all } \ell =1,2, \ldots.  \end{equation}
 		Recalling that the sets $D(\omega)$ and $ D_i(\omega)$ for $ i=1,2,\ldots,$ are uniformly bounded, and owing to (\ref{limsup_inclusion}), we have that, for all $j \in \N$, there exists $i_j$ such that for all $i \ge i_j$, $ D_i(\omega) \subset D^j(\omega) \ . $ Then for $q \in \R^K$ and for any Borel subset $B$ of $\Omega$, for all $i \ge i_j$, we have
 		\begin{align}\label{inequality_eta_i_maximum} \nonumber
 		q \cdot \int_B d \eta_i (\omega) & =  q \cdot \int _B \gamma_i (\omega) \ d \mu_i (\omega)  \nonumber = q \cdot \int_{B \cap \text{dom }D^j(.)} \gamma_i (\omega) \ d\mu_i(\omega)  \\& \leq  \int _B \sigma_q (\omega)\ d \mu_i (\omega) \leq  \int _B \psi_q^\ell (\omega)\ d \mu_i (\omega) \ .
 		\end{align}
 		The last inequality is a consequence of (\ref{lowerSemicontinuous_continuous}). Before passing to the limit, we observe that
 		\begin{equation} \label{support_property} \text{supp}(\eta) \subset \text{dom }D^j(.) \ .   \end{equation}
 		Indeed, take any open set $E \subset \Omega \setminus \text{dom }D^j(.)$. Since $\text{supp}(\eta_i) \subset \text{dom }D^j(.)$ for $i$ sufficiently large, and for all $j$, from (\ref{(d)_portmanteau}) of Theorem \ref{portmanteauTheorem}, we have 
 		\[ 0 \le \int_E d \eta_k^+ (\omega) \le \liminf\limits_{i} \int_E d \eta_{i,k}^+ (\omega) \le 0  \ . \]
 		We deduce that $\eta_k^+(E)=0$ for all $k= 1,\ldots,K $.
 		Following the same reasoning, one can conclude that $\eta_k^-(E)=0$ for all $k\in 1,\ldots,K$. Hence, $\eta(E)=0$ for all open subsets $E \subset \Omega \setminus \text{dom }D^j(.)$ and $\text{supp}(\eta) \subset \text{dom }D^j(.)$. The inclusion (\ref{support_property}) is therefore proved. 
 		By passing to the limit in (\ref{inequality_eta_i_maximum}) as $i \rightarrow \infty$, since $\psi_q^\ell(.)$ is bounded continuous on $\Omega$, we obtain for any Borel set $B \in B_{\eta,\mu}$
 	\[  q \cdot \int_B d \eta (\omega) \leq \int_B \psi^\ell_q(\omega)\ d \mu (\omega) \ . \]
 	As $\int_B d \eta(\omega) = \int_B \gamma(\omega) \ d \mu (\omega)$, for any $B \in B_{\eta,\mu}$, we have
 	\begin{equation} \label{inequality_psi_function}  q \cdot \int_B \gamma(\omega) \ d \mu (\omega) \leq \int_B \psi^\ell_q(\omega)\ d \mu (\omega) \ . \end{equation}
 	Recalling that $B_{\eta,\mu}$ generates the Borel $\sigma-$algebra $\mathcal B_\Omega$, we deduce that (\ref{inequality_psi_function}) is actually valid for all Borel subsets of $\Omega$. As a consequence,
 	$q \cdot \gamma(\omega) \leq  \psi^\ell_q(\omega) \quad \mu-\textrm{a.e. }, $ and letting $\ell \rightarrow \infty$, we obtain
 	\begin{equation} \label{inequality_gamma_and_continuousFunction} q \cdot \gamma(\omega) \leq \sigma_q(\omega) \quad \mu-\textrm{a.e.}  \end{equation}
 	Inequality (\ref{inequality_gamma_and_continuousFunction}) holds for all $q\in \R^K$ with $|q|=1$. (Indeed, from the continuity of the map
 		$ q \mapsto \max\{  q \cdot d : d \in D^j(\omega) \}$, it is enough to establish inequality (\ref{inequality_gamma_and_continuousFunction}) for $q \in \mathbb{Q}^K$, and subsequently use the density of $\mathbb Q^K$ in $\R^K$.)	

Since $D^j(\omega)$ is a closed and convex set, for each $\omega \in \text{dom } D(.)$, invoking the Hahn-Banach separation theorem, we obtain that
	\[ \gamma(\omega) \in D^j(\omega)   \quad \mu-\textrm{a.e.}\] Taking the limit as $j \rightarrow \infty$, we deduce that $\gamma(\omega) \in \bigcap\limits_{j \in \N} D^j(\omega) = D(\omega) \quad \mu-$a.e.  $\omega \in \Omega$ which concludes the proof.

\end{proof}

\section{Proofs of Theorem \ref{theorem_nco_general_case_probability_space}, Theorem \ref{atomic} and Theorem \ref{proposition: smooth data nco}} \label{section_proof_nco_general_case_probability_space}

We first employ a standard hypotheses reduction argument establishing that we can, without loss of generality, replace assumptions (\ref{A2_nco_general_case})-(\ref{A5_nco_general_case}) by the stronger conditions in which $\delta=+\infty$ (i.e. the conditions are satisfied globally).
\begin{enumerate}
	\item[(A3)$'$] \customlabel{A2'_nco_general_case}{(A3)$'$} There exist a constant $c>0$ and an integrable function $k_f:[0,T] \rightarrow \R$ such that
	\[ |f(t,x,u,\omega)-f(t,x',u,\omega)| \le k_f(t)|x-x'| \quad  \text{and} \quad |f(t,x,u,\omega)| \le c    \] for all $x,x' \in \R^n, \ u\in U(t), \ \omega \in \Omega, \ \text{a.e. } t \in [0,T] \ .$

	\item[(A4)$'$] \customlabel{A4'_nco_general_case}{(A4)$'$}
	\begin{enumerate}[label= (\roman{*}) , ref= (A4)$'$(\roman{*})]
		\item \label{A4'_ii_nco_general_case} There exist positive constants $k_g\ge 1$ and $M$ such that for all $\omega \in \Omega$\\
		$|g(x,\omega)| \leq M$ and  $d_{C(\omega)}(x)\leq M$  for all $x \in \R^n$,\vskip1ex
		$ |g (x, \omega) - g(x', \omega)| \leq k_g |x-x'|$  for all $x, x' \in \R^n \ .$ \item \label{A4'_iii_nco_general_case} There exists a modulus of continuity $\theta(.)$ such that we have
		\[|g(x,\omega_1)-g(x,\omega_2)| \leq \theta(\rho_\Omega (\omega_1,\omega_2)) 
		\]
		and
		$$
		|d_{C(\omega_1)}(x)-d_{C(\omega_2)}(x)| \leq \theta(\rho_\Omega (\omega_1,\omega_2)) \ ,
		$$
		for all $\omega_1,\omega_2 \in \Omega$ and $x \in \R^n$.
	\end{enumerate}

	\item[(A5)$'$] \customlabel{A5'_nco_general_case}{(A5)$'$} There exists a modulus of continuity $\theta_f(.)$ such that for all $\omega_1, \omega_2 \in \Omega$, \[ \int_{0}^{T} \sup\limits_{u \in U(t) , \  x \in \R^n}
	|f(t,x,u,\omega_1) - f(t,x,u,\omega_2) | d t \leq \theta_f (\rho_{\Omega}(\omega_1, \omega_2)). \]
	\end{enumerate}
	
\noindent	
This is possible if we consider the ``truncation'' function $tr_{y,\delta}:\R^n \rightarrow \R^n$, defined to be
\[ tr_{y,\delta}(x) := \begin{cases} x \quad & \text{ if } |x-y| < \delta \\  y + \delta \frac{x-y}{|x-y|}  & \text{ if } |x-y| \ge \delta \ ,  \end{cases}      \] and we replace $f, \ g$ and $d$ above by their local expression $\widetilde{f}$, $\widetilde{g}$ and $\widetilde{d}$  defined as follows
\[ \widetilde f (t,x,u,\omega) := f(t,tr_{\bar x(t,\omega),\delta}(x),u,\omega), \quad 
\widetilde g (x,\omega) := g(tr_{\bar x(T,\omega),\delta}(x);\omega), \quad 
\widetilde d _{C(\omega)}(x):= d _{C(\omega)}(tr_{\bar x(t,\omega),\delta}(x))  \ .    \]

\noindent
Indeed, the problems involving the functions $(f,g, d)$ and $(\widetilde f , \widetilde g, \widetilde d )$ do coincide in a neighbourhood of the $W^{1,1}-$local minimizer $(\bar u, \{ \bar x (.,\omega)   \ | \ \omega \in \Omega \})$ for (\ref{intprob}). Therefore, $(\bar u, \{ \bar x (.,\omega) \ | \ \omega \in \Omega \})$ does remain a  $W^{1,1}-$local minimizer for the problem (\ref{intprob}) when we substitute the pair $(f,g,d)$ with $(\widetilde f, \widetilde g, \widetilde d )$.  Furthermore, the assertions of the theorem are unaffected by changing the data in this way.

\noindent
We provide two technical lemmas which will be employed in the approximation techniques used in the theorems proof. These preliminary results establish the uniform continuity of trajectories with respect to $\omega$ and the existence of a sequence of suitable finite support measures approximating the reference measure $\mu$. 
Throughout this section, $d_\mathbcal{E}(.,.)$ denotes the Ekeland metric defined on the control set $\mathcal{U}$ as
\[ d_\mathbcal{E} (u_1,u_2) := \textrm{meas } \{ t \in [0,T] \ | \ u_1(t) \neq u_2(t) \} . \]
We recall that, given a control $u(.)$, to make clearer which control is used we shall employ the alternative notation $x(.,u,\omega)$ for the feasible arc belonging to the family of trajectories $\{ x(.,\omega) \ : \ \omega \in \Omega  \}$ associated with the control $u(.)$.
\\
\begin{lemma}\label{lemma_continuity_omega_control} Let $(\Omega,\rho_\Omega)$ be a metric space. Suppose that assumptions (\ref{A1_nco_general_case})(i)-(ii), \ref{A2'_nco_general_case} and \ref{A5'_nco_general_case} are satisfied. Then,
	\begin{enumerate}[label= (\roman{*})]
		\item[(i)] \customlabel{lemma: uniformity w.r.t. control}{Lemma \ref{lemma: property uniform continuity w.r.t the control}(i)} 
		 we can find $\beta >0$ such that 
		\begin{equation}\label{lemma: property uniform continuity w.r.t the control} 
		\sup\limits_{\omega \in \Omega} \{ \| x(.,u,\omega) - x(.,u',\omega) \|_{L^\infty} \} \le \ \sup\limits_{\omega \in \Omega} \{ \| x(.,u,\omega) - x(.,u',\omega) \|_{W^{1,1}} \} \le \ \beta d_{\mathbcal E}(u,u'), 
		\end{equation}
				for all $u(.), u'(.) \in \mathcal{U}$.
		\item[(ii)]\customlabel{lemma: unformity w.r.t. parameter}{Lemma \ref{lemma: property uniform continuity w.r.t the control}(ii)} for all $\widetilde \varepsilon >0$, we can find $\widetilde r >0$, such that for any given $u(.) \in \mathcal{U}$,
		\[\| x(.,u,\omega) - x(.,u,\omega') \|_{L^\infty}  < \widetilde\varepsilon \ \text{ for all }\ \omega,\omega' \in \Omega \ \text{ such that } \rho_\Omega (\omega,\omega')   < \widetilde r  \ .\] 
	\end{enumerate}	
\end{lemma}

\begin{proof}
(i) Write
\begin{equation*} \label{r_choice} \beta \ := \ 2c \exp (\int_0^T k_f(s) ds) \ . \end{equation*} 
Fix any $ \varepsilon >0$. Take any $u(.),u'(.) \in \mathcal{U}$. 
Owing to Filippov Existence Theorem \cite[Theorem 2.4.3]{vinter2010optimal} (recall that we have the same initial datum $x_0$), for each $\omega \in \Omega$, we obtain

\begin{align*} 
	\| x(.,  u,\omega)  -  x(.,u',\omega)  \|_{L^\infty}  & \le \| x(.,  u,\omega)  -  x(.,u',\omega)  \|_{W^{1,1}}  \\
		& \le \exp \left(\int_0^T k_f(s) ds \right) \int_0^T \left| f(t, x (t,u',\omega),u'(t),\omega)  - f(t, x (t,u',\omega),u(t),\omega)   \right| \ dt \\ & \le 2c \exp \left(\int_0^T k_f(s) ds \right) d_{\mathbcal E}(u,u')   .\end{align*}
The last inequality is a consequence of the bound on the dynamic (assumption \ref{A2'_nco_general_case}).
The particular choice $\beta$ allows to conclude. 

\noindent
(ii) Fix now any $\widetilde \varepsilon >0$. Take a control $u(.)\in \mathcal U$.
Owing to assumption \ref{A5'_nco_general_case}, we choose $\widetilde r>0$ such that 
\begin{equation} \label{theta inequality}  \theta_f(r')  \le \frac{\widetilde \varepsilon}{\exp\left(\int_{0}^{T} k_f(s)ds\right)} \quad \text{ for all } 0 < r' \le \widetilde r \ .   \end{equation}
Take $\omega, \ \omega' \in \Omega$ such that $\rho_\Omega(\omega,\omega') < \widetilde r.$ 
Taking two different trajectories $x(.,u,\omega)$ and $x(., u, \omega')$ with the same initial point $x_0$ and the same control $u(.)$, for all $t\in [0,T]$ we have,
\begin{align}  \nonumber & 
|x(t,u(t),\omega)   -   x(t, u(t), \omega') | \nonumber \leq \int_0^t |f(s,x(s,u(s),\omega),u(s),\omega) - f(s,x(s,u(s),\omega'),u(s),\omega') | \ d s
\\ & \hspace{-0.3cm} \nonumber \leq \int_0^t |f(s,x(s,u(s),\omega),u(s),\omega) - f(s,x(s,u(s),\omega'),u(s),\omega) | \ d s \\ & \qquad + \int_0^t |f(s,x(s,u(s),\omega'),u(s),\omega) - f(s,x(s,u(s),\omega'),u(s),\omega') | \ d s \ .
\end{align}
Taking into account assumptions \ref{A2'_nco_general_case} and \ref{A5'_nco_general_case}, we conclude that	
\begin{align*} \label{z(t)_gronwall_estimate} \nonumber
|x(t,u(t),\omega) &  -   x(t, u(t), \omega') | \leq \int_0^t k_f(s) |x(s,u(s),\omega)-x(s,u(s),\omega')| ds + \theta_f(\rho_\Omega(\omega,\omega')) \ .\end{align*}

\noindent Applying Gronwall Lemma, for all $t\in [0,T]$, we deduce
\[   |x(t,u(t),\omega)   -   x(t, u(t), \omega')| \leq  \theta_f(\rho_\Omega(\omega,\omega')) \exp \left(\int_0^t k_f(s) \textrm{d}s \right)  \ .  \]
The particular choice of $\widetilde r$ as in (\ref{theta inequality}) and the fact that $\rho_{\Omega}(\omega, \omega') < \widetilde{r}$ allow to conclude the proof.

\end{proof}

\begin{lemma} \label{lemma_relation_sequence_dirac_measure_with_given_probability}
	Suppose that conditions (\ref{A3_nco_general_case}), (\ref{A1_nco_general_case})(i)-(ii), \ref{A2'_nco_general_case}-\ref{A5'_nco_general_case} are satisfied, and $\mu \in \mathcal{M}(\Omega)$. Then, there exist a sequence of finite subsets of $\Omega$, $\{  \Omega^\ell := \{ \omega_j^\ell \ : \ j=0,\ldots,N_\ell \}   \}_{\ell\ge 1}$ and a sequence of convex combinations of Dirac measures $\{ \mu_\ell \}_{\ell \ge 1}$, such that the following properties are satisfied. 
	\begin{enumerate}[label=(\roman*), ref=\roman{*}]
		\item \label{proof: (i) result on countable sets} $\Omega^\ell \subset \Omega^{\ell+1}$ for all integer $\ell \ge 1$, and $\widehat \Omega := \bigcup\limits_{\ell \ge 1} \Omega^\ell$ is a countable dense subset of $\text{supp}(\mu)$;
		\item \label{proof: (ii) result on countable sets} $\mu_\ell = \sum_{j=0}^{N_\ell} \alpha_j^\ell   \delta_{\omega_j^\ell}$, where $\alpha_j^\ell \in (0,1]$ and $\sum_{j=0}^{N_\ell} \alpha_j^\ell =1$, and $\mu_\ell  \stackrel{*}\rightharpoonup \mu $ ;
		\item \label{proof: (iii) result on approximation of the cost function g}for each $\varepsilon>0$, we can find $\ell_\varepsilon \in \N$ such that for all $\ell \ge \ell_\varepsilon$,
		\[ \left| \int_{\Omega} g(x(T,u,\omega);\omega)\ d \mu_\ell -  \int_{\Omega} g(x(T,u,\omega);\omega) \ d \mu \right| \le \varepsilon 
		\]
		and
		\[ \left| \int_{\Omega} d_{C(\omega)}(x(T,u,\omega)) \ d \mu_\ell -  \int_{\Omega} d_{C(\omega)}(x(T,u,\omega))  \ d \mu \right| \le \varepsilon  \]
		for all $u(.) \in \mathcal{U}$.
	\end{enumerate}
	Moreover, if the measure $\mu$ has a purely atomic component such that each atom is a singleton, then the countable set $\widehat \Omega$ can be constructed in such a manner that $\widehat \Omega$ contains all the atoms of $\mu$.
	
\end{lemma}

\begin{proof}
	(i). Since $\Omega$ is a complete separable metric space, the measure $\mu$ is tight. As a consequence, for all integer $\ell \ge 1$, there exists a compact set $K_\ell \subset \Omega$ such that
		$  \mu(\Omega \setminus K_\ell)  < \frac{1}{\ell}.   $
	Write $\Omega_0^\ell := (\Omega \setminus K_\ell) \cap \text{supp}(\mu)$. Therefore, employing an iterative argument, a suitable choice of the compact set $K_\ell$ allows us to obtain, for each $\ell \ge 1$, a family of disjoint Borel subsets $\{\Omega_j^\ell\}_{j=0,\ldots,N_\ell}$, for some $N_\ell \in \N$, such that the following conditions are satisfied:
	\begin{enumerate}[label=(\alph{*}), ref=\alph*]
		\item $ \text{supp}(\mu) = \bigcup\limits_{j=0}^{N_\ell} \Omega_j^\ell $;
		\item for each $j\in \{ 1,\ldots,N_\ell  \}$, ${\Omega}_j^\ell \subset K_\ell$, and diam$(\Omega_j^\ell) \le \frac{1}{\ell}$. 
		(Recall that diam$(\Omega_j^\ell)= \sup\limits_{a,b \in \Omega_j^\ell} \rho_\Omega(a,b)$.)
		\item $\mu(\Omega_0^\ell)  < \frac{1}{\ell}$ \ and \ $\Omega_0^\ell \supset  \Omega_0^{\ell+1}$.
	\end{enumerate}
		We can also choose elements $\omega_j^\ell \in \Omega_j^\ell$, for all $j=0,1,\ldots,N_\ell$, in such a manner that we have $\{ \omega_j^\ell \}_{j=0,\ldots,N_\ell} \subset \{ \omega_j^{\ell+1} \}_{j=0,\ldots,N_{\ell+1}}$. If $\text{supp}(\mu)$ is compact, then we can always assume that $\Omega_0^\ell = \emptyset$ for all integer $\ell \ge 1$. In this case, we can relabel the elements chosen in the Borel sets $\Omega_j^\ell$'s, taking
		\[ \omega_j^\ell   \in \Omega_{j+1}^\ell \ ,\qquad \text{for all } j=0,1,\ldots, N_\ell -1   \] and we replace $N_\ell$ with $\widetilde{N}_\ell:= N_\ell -1$. In any case, we obtain, for each $\ell \ge 1$, a finite set $\Omega^\ell := \{ \omega_j^\ell \}_j$ such that $\Omega ^ \ell \subset \Omega^{\ell +1}$.
		\noindent
		From the standard properties of complete separable metric spaces, it is easy to see that the sequence of sets $\{ \Omega^\ell \}$ can be constructed in such a way that $\widehat \Omega := \bigcup\limits_{\ell \ge 1} \Omega^\ell$ is (countable) dense in $\text{supp}(\mu)$.
		\ \\
				
\noindent	
(ii).	We assume here that $\text{supp} (\mu)$ is not compact (the compact case can be treated in a similar and easier way).  
		Consider, for each $\ell \ge 1$, the family of Borel disjoint subsets of $\Omega$, $\{ \Omega_j^\ell  \}_{j=0,\ldots,N_\ell}$ and the finite sequence of elements $\{\omega_j^\ell\}_{j=0,\ldots,N_\ell}$, with $\omega_j^\ell \in \Omega^\ell_j$, provided in the proof of (i). We define the measure $\mu_\ell$
		\[ \mu_\ell:= \sum_{j=0}^{N_\ell} \mu(\Omega^\ell_j) \delta_{\omega^\ell_j}  \ . \]
		Owing to Theorem \ref{portmanteauTheorem}, we can check the weak$^*$ convergence of the sequence $\{ \mu_\ell \}$ on the set of bounded real valued uniformly continuous functions on $(\Omega,\rho_\Omega)$ (instead of the set of bounded continuous functions). Take any bounded uniformly continuous function $h:\Omega\to \R$. Write $M:=\sup\limits_{\omega\in \Omega} |h(\omega)|$. 
		Fix any $\varepsilon > 0$. Then, there exists $r_\varepsilon>0$ such that
		\begin{equation} \label{inequality on h} |h(\omega_1) - h(\omega_2)| \leq \frac{\varepsilon}{6} \quad \text{ for all } \omega_1,\omega_2 \in \Omega \quad \textrm{with } \rho_\Omega(\omega_1,\omega_2) \leq r_\varepsilon \ . \end{equation}
		Let $\ell_\varepsilon \in \N$ such that $\frac{1}{\ell_\varepsilon} \le \min \{r_\varepsilon; \frac{\varepsilon}{4M}\}$. Then for all $\ell \ge \ell_\varepsilon$, we have
		\begin{align} \label{equation: integral of h proof of Lemma}\int_{\Omega} h \ d \mu_\ell - \int_{\Omega} h \ d\mu & = \sum_{j=0}^{N_\ell} \mu (\Omega^\ell_j) h(\omega^\ell_j) - \sum_{j=0}^{N_\ell} \int_{\Omega^\ell_j} h(\omega) \ d \mu(\omega) = \sum_{j=0}^{N_\ell} \int_{\Omega^\ell_j} h(\omega^\ell_j) - h(\omega)\ d\mu(\omega) \ . \end{align}
		For each $j\in \{1,\ldots,N_\ell\}$, we define
		\[\beta_j^\ell := \inf_{\omega \in \Omega_j^\ell} h(\omega) \qquad \text{and} \qquad \gamma_j^\ell := \sup_{\omega \in \Omega_j^\ell} h(\omega)   \ . \]
		Therefore, we can find $y_j^\ell, \ z_j^\ell \in \Omega_j^\ell$ such that
		\begin{align*}
		& h(y_j^\ell) \le \beta_j^\ell + \frac{\varepsilon}{6} \qquad \text{and} \qquad h(z_j^\ell) \ge \gamma_j^\ell - \frac{\varepsilon}{6} \ .
		\end{align*}
		Then for all $\ell \ge \ell_\varepsilon$, using also (\ref{inequality on h}) and the fact that $\text{diam}(\Omega_j^\ell) \leq \frac{1}{\ell}$, it follows that
		\begin{align} \label{beta_gamma_difference1}  \gamma_j^\ell - \beta_j^\ell  \le h(z_j^\ell) - h(y_j^\ell) + \frac{2}{6} \varepsilon \leq \frac{\varepsilon}{2}, \quad \mbox{for all} \;\; j = 1,\ldots,N_\ell \ .      \end{align}
		As a consequence, for all $\ell \ge \ell_\varepsilon$, from (\ref{equation: integral of h proof of Lemma}) we deduce that
		\begin{align*} \left| \int_{\Omega} h d \mu_\ell - \int_{\Omega} h d \mu \right| & \le  \sum_{j=0}^{N_\ell}  \int_{\Omega^\ell_j} \left|h(\omega^\ell_j) - h(\omega)\right| d\mu(\omega) \\ & \le  \sum_{j=1}^{N_\ell}  \int_{\Omega^\ell_j} \quad \left( \sup\limits_{\omega' \in \Omega_j^\ell}  h(\omega') - \inf\limits_{\omega'' \in \Omega_j^\ell} h(\omega'') \right) d\mu(\omega)
			+  \int_{\Omega^\ell_0} \left|h(\omega^\ell_0) - h(\omega)\right| d\mu(\omega)\ . \end{align*}
		Then, from inequality (\ref{beta_gamma_difference1}) and the choice of $\ell_\varepsilon$, 
		for all $\ell \ge \ell_\varepsilon$, we obtain
		\begin{align*} \left| \int_{\Omega} h d \mu_\ell - \int_{\Omega} h d \mu \right| & \le  \sum_{j=1}^{N_\ell} \int_{\Omega^\ell_j} (\gamma_j^\ell  - \beta_j^\ell) \textrm{d}\mu + 2 M \mu(\Omega^\ell_0) 
		\le \frac{\varepsilon}{2} + \frac{\varepsilon}{2}  \le \varepsilon\ . \end{align*}
		Setting $\alpha_j^\ell:= \mu(\Omega^\ell_j) >0$, for $j=0,\ldots,N_\ell$, we conclude the proof of (\ref{proof: (ii) result on countable sets}).
		
\ \\		
\noindent	
(iii). Fix any $\varepsilon > 0$. Choose $r_0 >0$ such that
		\[\theta(r) \leq \frac{\varepsilon}{4} \qquad \text{for all} \quad 0<r \le r_0 \ . \]
		Take any $\omega_1, \ \omega_2 \in \Omega$ such that $\rho_\Omega (\omega_1,\omega_2) <  r_0$. Then, from assumption \ref{A4'_iii_nco_general_case}
		\[  |g(x,\omega_1)-g(x,\omega_2)| < \frac{\varepsilon}{4}\quad 
		\text{and} \quad |d_{C(\omega_1)}(x)-d_{C(\omega_2)}(x)| < \frac{\varepsilon}{4}  \qquad \text{for all } x \in \R^n \ . \]
		Take any $u(.) \in \mathcal U$. From Lemma \ref{lemma_continuity_omega_control}(ii), there exists $\widetilde r>0$ such that for all $\omega_1,\omega_2 \in \Omega$ verifying $\rho_\Omega (\omega_1,\omega_2)< \tilde r$, we have
		\[  |x(t,u,\omega_1) - x(t,u,\omega_2)|  \leq \frac{\varepsilon}{4k_g} \quad \text{for all } t\in [0,T] 
		\ . \]
		Write $ r_\varepsilon := \min \{\widetilde r, r_0\}$. For all $\omega_1,\omega_2 \in \Omega$ verifying $\rho_\Omega(\omega_1,\omega_2) \le r_\varepsilon$, from assumption \ref{A4'_ii_nco_general_case}, we deduce
		\begin{align*} &  |g(x(T,u,\omega_1);\omega_1) -  g(x(T,u,\omega_2);\omega_2)| \\ & \le |g(x(T,u,\omega_1);\omega_1) -  g(x(T,u,\omega_1);\omega_2)| + |g(x(T,u,\omega_1);\omega_2) -  g(x(T,u,\omega_2);\omega_2)| \\ & \le \frac{\varepsilon}{4} + k_g |x(T,u,\omega_1) - x(T,u,\omega_2)| = \frac{\varepsilon}{2}
		\ .  \end{align*} 
		Similarly, $|d_{C(\omega_1)}(x(T,u,\omega_1)) -  d_{C(\omega_2)}(x(T,u,\omega_2))|\le \frac{\varepsilon}{2}$. 
		Therefore, for each $u(.) \in \mathcal U$, the maps $\omega \mapsto g(x(T,u,\omega);\omega)$ and $\omega \mapsto d_{C(\omega)}(x(T,u,\omega))$ are uniformly continuous, and from \ref{A4'_nco_general_case} (uniformly) bounded by the constant $M$ (observe that $M$ and $r_\varepsilon$ above do not depend on $u(.)$). Invoking the same argument employed in the proof of (\ref{proof: (ii) result on countable sets}) we conclude that, whenever we fix $\varepsilon > 0$, we can find $\ell_\varepsilon\in \N$ such that for all $\ell \ge \ell_\varepsilon$, we have
		\begin{align*} & \left| \int_{\Omega} g(x(T,u,\omega);\omega) \ d \mu_\ell(\omega) - \int_{\Omega} g(x(T,u,\omega);\omega)\ d \mu(\omega) \right| \le \varepsilon \  \end{align*}
		and
		\[ \left| \int_{\Omega} d_{C(\omega)}(x(T,u,\omega)) \ d \mu_\ell -  \int_{\Omega} d_{C(\omega)}(x(T,u,\omega))  \ d \mu \right| \le \varepsilon  \ . \]
		This confirms property (\ref{proof: (iii) result on approximation of the cost function g}).
		
	\ \\
				
\noindent	
Finally, if the measure $\mu$ has a purely atomic component such that each atom is a singleton, then at each step of the iterative argument employed in (i), the compact set $K_\ell \subset \Omega$, for all $\ell \ge 1$, is such that it contains a finite number of atoms of $\mu$ which will be included in $\Omega^\ell$.

\end{proof}
\noindent
{\bf Proof of Theorem \ref{theorem_nco_general_case_probability_space}}. The proof is build up in four parts. The first part consists in approximating the reference problem with a given probability measure by an auxiliary problem which involves measures with finite support. This is possible invoking the result on the weak$^*$ convergence established in Lemma \ref{lemma_relation_sequence_dirac_measure_with_given_probability} and the Ekeland's variational Principle. In the second part, we apply necessary optimality conditions (cf. Proposition \ref{finitenco} previously obtained) for the auxiliary problem. In the third part, we pass to the limit a first time to obtain optimality conditions on a countable dense subset of supp$(\mu)$. The last part of the proof is devoted to deriving, via a second limit-taking process, all the desired necessary conditions of the theorem statement. Since it is not restrictive to assume that supp$(\mu) = \Omega$, we shall consider this assumption throughout the proof. 
\vskip2ex\noindent
{\bf 1.} Take a $W^{1,1}-$local minimizer $(\bar u,\{ \bar x (.,\omega) \ : \ \omega \in \Omega \})$ for problem (\ref{intprob}). Then there exists $\bar \varepsilon >0$ such that
\[ \int_{\Omega} g(\bar x (T,\omega);\omega) d \mu (\omega) \le \int_{\Omega} g(x (T,\omega);\omega) d \mu (\omega)     \] for all feasible processes $(u,\{x(.,\omega) \ : \ \omega \in \Omega \})$ such that
\[  \| \bar x (.,\omega)  - x(.,\omega)  \|_{W^{1,1}} \le \bar \varepsilon \quad \text{for all} \quad \omega \in  \Omega \ (=\text{supp}(\mu)) \ . \]
Take a decreasing sequence $\epsilon_i \downarrow 0$ such that $\beta \epsilon_i\le \frac{\bar \varepsilon }{4}$ for all $i\ge 1$, where $\beta>0$ is the number provided by Lemma \ref{lemma_continuity_omega_control}.
For each $i$, we define the functional $J_i : ({\mathbcal U},d_{\mathbcal E}) \rightarrow \R$ as follows:
\begin{align*} & J_i(u)  := \bigg[ \int_{\Omega} \bigg( g( x (T,u,\omega);\omega) - \int_{\Omega} g(\bar x (T,\omega);\omega) \ d \mu (\omega) + \epsilon_i^2  \bigg) \ d\mu(\omega)  \bigg]  \vee \int_{\Omega} d_{C(\omega)}(x(T,\omega)) \ d\mu(\omega) .  \end{align*} 
It is clear that $J_i(u) \ge 0$, for all controls $u(.)$. Moreover, we have $J_i(u) > 0$ for all controls $u \in \mathcal{U}_{\bar \varepsilon }$, where
\begin{align*} \mathcal{U}_{\bar \varepsilon }:= \{u(.) \in \mathcal{U} \; : & \; \text{ the associated process } (u(.),  \{ x(.,\omega)  \ : \ \omega \in \Omega  \})    \\ & \text{ satisfies  } \| x(.,\omega) - \bar x (.,\omega) \|_{W^{1,1}} \le \bar \varepsilon \ \text{ for all } \omega \in \Omega \}. \end{align*}   
Otherwise, there would exist $\hat{u} \in \mathcal{U}_{\bar \varepsilon }$ such that $J_{\Omega}((\hat u(.), \{\hat x(.,\omega)\}))<J_{\Omega}((\bar u(.), \{\bar x(.,\omega)\}))$, contradicting the fact that $(\bar u, \{\bar x(.,\omega) \ : \ \omega \in \Omega   \})$ is a $W^{1,1}-$local minimizer for (P). Observe also that 
\[ J_i(\bar u) \le \inf\limits_{u \in \mathcal{U}}  J_i(u) + \epsilon_i^2 \ ,  \] 
which means that $\bar u$ is an $\epsilon_i^2-$minimizer for $J_i$ on ${\mathbcal U}$. Then, since $J_i$ is a continuous function on the complete metric space $({\mathbcal U},d_{\mathbcal E})$ (it suffices to use here the Lipschitz continuity of $g(.,\omega)$ and $d_{C(\omega)}(.)$ and Lemma \ref{lemma_continuity_omega_control}(i)), we deduce from Ekeland's Theorem (cf. \cite[Theorem 3.3.1]{vinter2010optimal}) that, for each $i\ge 1$, there exists $v_i \in {\mathbcal U}$ such that
	\begin{equation} \label{estimation_controls_ekeland_metric} d_{\mathbcal E}(v_i,\bar u) \le \epsilon_i \quad \text{and}  \end{equation}
	\begin{equation}  J_i(v_i)+\epsilon_i d_{\mathbcal E}(v_i,v_i) = \min\limits_{u \in  {\mathbcal U}} \{  J_i(u)+\epsilon_id_{\mathbcal E}(u,v_i)   \} \ .  \end{equation}
Consider the sequence of convex combinations of Dirac measures $\{  \mu_\ell \}$ provided by Lemma \ref{lemma_relation_sequence_dirac_measure_with_given_probability}. Recall, in particular, that $\mu_\ell \stackrel{*}\rightharpoonup \mu$ and 
\[ \mu_{\ell} = \sum_{j=0}^{N_\ell} \alpha_j^\ell \delta_{\omega_j^\ell}  \] where $\alpha_j^\ell \in (0,1]$, for all $j=0,\ldots,N_\ell$ and $\sum_{j=0}^{N_\ell} \alpha_j^\ell=1$. 
	We can find a decreasing sequence $\rho_i \downarrow 0$, with $\beta \rho_i\le \frac{\bar \varepsilon }{4}$ for all $i\ge 1$, and an increasing sequence $\{\ell_i \in \N\}_{i\ge 1}$ such that, setting
\begin{align*} & \widetilde{J}_i(u)  := \bigg[ \int_{\Omega^i} \bigg( g( x (T,u,\omega);\omega) - \int_{\Omega} g(\bar x (T,\bar u, \omega);\omega) \ d \mu (\omega) + \epsilon_i^2   \bigg) \ d\mu_i(\omega)  \bigg]  \vee \int_{\Omega^i} d_{C(\omega)}(x(T,u,\omega)) \ d\mu_i(\omega)  \end{align*}
(we write $\Omega^i := \Omega^{N_{\ell_i}} \subset \widehat{\Omega} \subset \Omega$, $\mu_i := \mu_{\ell_{i}}$ for the corresponding convex combination of $(N_i+1 = N_{\ell_i}+1)$ Dirac measures which approximate $\mu$, and $\omega_j^i := \omega_j^{\ell_i}$, $j=0,1,\ldots,N_i$, so that $\Omega^i = \{ \omega_j^i  \}_{j=0}^{N_i}$), we have  
\[  \big|\widetilde{J}_i(u)  -  {J}_i(u) \big| \le \frac{\rho_i^2}{2} \qquad \text{for all } u \in {\mathbcal U} \]
and
\begin{equation} \label{J1}  
\widetilde{J}_i(u) \ge {J}_i(u) - \frac{\rho_i^2}{2}>0 \qquad \text{for all } u \in {\mathbcal U}_{\bar \varepsilon }. 
 \end{equation}
Therefore, $v_i$ is a $\rho_i^2-$minimizer on ${\mathbcal U}$ for
\[  u \to \widetilde{J}_i(u)  + \epsilon_i d_{\mathcal{E}}(u,v_i) \ .  \]
Invoking Ekeland's theorem one more time, we deduce that there exists $u_i \in {\mathbcal U}$ which minimizes
\begin{equation} \label{minimizing_property_ekeland} u \to \widetilde{J}_i(u)  + \epsilon_i d_{\mathcal{E}}(u,v_i) + \rho_i  d_{\mathcal{E}}(u_i,u)  \quad \text{on } {\mathbcal U} \end{equation}
such that $d_{\mathcal{E}}(u_i,v_i) \leq \rho_i$. As a consequence we obtain
\begin{equation}\label{rhoi}
d_{\mathcal{E}}(u_i,\bar u) \le \epsilon_i + \rho_i=:\rho_i' \ .
\end{equation}
Write $(u_i,\{ x_i(., \omega) \ : \ \omega \in \Omega \})$ the process associated with the control $u_i$. 
Therefore, from Lemma \ref{lemma_continuity_omega_control} (i) we have that
\begin{equation} \label{xi}
\sup\limits_{\omega \in \Omega} \{ \| x_i(.,\omega) - \bar x(.,\omega) \|_{L^\infty} \} \le \ \sup\limits_{\omega \in \Omega} \{ \| x_i(.,\omega) - \bar x(.,\omega) \|_{W^{1,1}} \} \le \ \beta \rho_i' \ (\le \bar \varepsilon / 2) \ .
\end{equation}
Bearing in mind (\ref{J1}) it immediately follows that $\widetilde{J}_i(u_i)>0$.
	
\noindent	
Now we introduce two $\mathcal{L} \times \mathcal{B}^m-$ measurable functions
\[ m_i(t,u) := \begin{cases} 0 \quad \text{if } u = v_i(t) \\ 1 \quad \text{otherwise}, \end{cases} \qquad \text{and} \qquad m_i'(t,u) := \begin{cases} 0 \quad \text{if } u = u_i(t) \\ 1 \quad \text{otherwise}. \end{cases} \]
Therefore we can write:
\[ d_{\mathbcal E}(u,v_i) = \int_{0}^{T}m_i(t,u(t)) \ dt \qquad \text{and} \qquad    d_{\mathbcal E}(u,u_i) = \int_{0}^{T}m_i'(t,u(t)) \ dt \ .  \]
The minimizing property (\ref{minimizing_property_ekeland}) can be expressed in terms of the following auxiliary optimal control problem
\leqnomode 
\begin{equation*} \begin{cases} \label{auxiliaryprob}
\begin{aligned}
& {\text{minimize}}
& & \widetilde J_{i}(u)  + \epsilon_i  \gamma(T) + \rho_i \zeta(T) \\
&&& \hspace{-1.9cm} \text{over controls } u(.) \in \mathcal{U} 
\text{ and family of } W^{1,1} \text{arcs }  \{ x(.,\omega) \} \text{ s.t. for all } \omega \in \Omega^i \\
&&& \dot{x}(t,\omega) = f(t,x(t,\omega), u(t), \omega) \quad \textrm{a.e. } t \in [0,T] \\
&&& \dot \gamma(t) = m_i(t, u(t)) \quad \textrm{a.e. } t \in [0,T] \\
&&& \dot \zeta(t) = m_i'(t, u(t)) \quad \textrm{a.e. } t \in [0,T] \\
&&& x(0,\omega)=x_0 \\
&&& \gamma(0)=0 \text{ and } \zeta(0)=0 \end{aligned} \tag{P$_i$} \end{cases}
\end{equation*}\reqnomode
whose minimizer is the family $(u_i, (\gamma_i , \zeta_i \equiv 0,\{ x_i(., \omega) \}) )$ verifying, as $i \rightarrow \infty$, $d_{\mathbcal E}(u_i,\bar u) \rightarrow 0 $
and \begin{equation} \label{trajectory_limit_auxiliary_problem}
\sup\limits_{ \omega \in \Omega} \| \bar x (., \omega) - x_i(.,\omega) \|_{W^{1,1}}  \rightarrow 0 \ .
\end{equation}
\vskip2ex\noindent
{\bf 2.} The second step of the proof consists in applying necessary optimality conditions (cf. Proposition \ref{finitenco}) to problem (\ref{auxiliaryprob}) for each $i$ sufficiently large: for all $\omega \in \Omega^i$ (that is for $\mu_i-$a.e. $\omega \in \Omega$), there exist $W^{1,1}-$arcs $ p_i(., \omega)$ (associated with the state variable $x$), $q_i(.)$ (associated with the variable $\gamma$), and $z_i(.)$ (associated with the variable $\zeta$) such that
	\begin{equation} \label{equation: nontriviality condition proof general case0} (p_i(.,\omega), q_i(.),z_i(.)) \neq (0,0,0) \ ,
		\end{equation}
		and satisfying the necessary conditions below:
		\ \\
	\noindent
	The transversality condition (owing to the Max Rule \cite[Theorem 5.5.2]{vinter2010optimal}), for suitable $\lambda_i \in [0,1]$, leads to
	\begin{align} \label{transversality_condition_auxiliary_problem} 
	-  p_i(T, \omega) \in \alpha_j^i \lambda_i  \partial_x g (x_i(T, \omega); \omega) + \alpha_j^i(1-\lambda_i) \partial d_{C(\omega)} (x_i(T, \omega)), \quad  -q_i(T) =  \epsilon_i  \quad \text{and} \quad  -z_i(T) =  \rho_i  
	\ .   \end{align}
	(Here, $\alpha^i_j:= \alpha^{\ell_i}_j$, for $j=0,1, \dots, N_i$.) 
	The adjoint system gives $  - \dot q_i(t) \equiv 0 $ and  $  - \dot z_i(t) \equiv 0 $, which implies that
	$q_i(t)\equiv - \epsilon_i \ ,$
	and $z_i(t)\equiv -  \rho_i \ .$ Moreover,
	\begin{equation} \label{adjoint_system_axiliary_problem} - \dot { p}_i (t, \omega) \in \textrm{co } \partial_x [ p_i(t,\omega) \cdot f(t, x_i(t,\omega), u_i(t), \omega)]  \quad \text{a.e. } t \in [0,T]  \ . \end{equation}
	From the maximality condition, we obtain, for a.e. $t\in [0,T]$
\begin{align*}  & \sum_{\omega \in \Omega^i} p_i(t,\omega) \cdot  f(t,x_i(t,\omega),u_i(t),\omega)   - \epsilon_i m_i(t,u_i) \\ &  = \max_{u \in U(t)} \left( \sum_{\omega \in \Omega^i}  p_i(t,\omega) \cdot f(t, x_i(t,\omega),u,\omega) -  \epsilon_i m_i(t,u) - \rho_i m_i'(t,u)
\right).  \end{align*}
	This implies that for a.e. $t \in [0,T]$ and for every $u \in U(t)$
\begin{align} \label{maximization_condition_with_integral_inequality}  \sum_{\omega \in \Omega^i} & p_i(t,\omega) \cdot [f(t, x_i(t,\omega),u,\omega) -  f(t,x_i(t,\omega),u_i(t),\omega) ] \leq  \rho_i' \ .  \end{align}

\noindent
From (\ref{rhoi})	we deduce that
	\[  u_i = \bar u(t) \qquad \text{on a set } A_{\rho_i'} \subset [0,T] \text{ such that } \text{meas}([0,T] \setminus A_{\rho'_i}) \le \rho_i' . \]
	\noindent 
	Moreover, taking note of the fact that $d_{\mathcal{E}}(u_i,\bar u) \le \rho_i'$ 
	and, owing to Lemma \ref{lemma_continuity_omega_control} (i), we can also deduce that 
	\begin{equation}  \label{convergence of arcs} x_i(t,\omega)   \in \bar x (t,\omega) + \beta\rho_i'\B 
	\quad \text{for all } \omega \in \Omega , \text{ and for all } t \in [0,T] \ .   \end{equation}
	  Therefore, for each $i$, and $\mu_i-$a.e. $\omega \in \Omega$, from the optimality conditions (\ref{equation: nontriviality condition proof general case0})-(\ref{maximization_condition_with_integral_inequality}), we have
	  \begin{enumerate}[label=(a\arabic*), ref=(a\arabic*)]
\item $ p_i(.,\omega) \neq 0$ \ ;
	  	\item $ - \dot { p}_i (t, \omega) \in \textrm{co } \partial_x [ p_i(t,\omega) \cdot f(t, x_i(t,\omega), \bar u(t), \omega)]$ \quad for all $t \in A_{\rho_i'}$ ;
	  	\item $-  p_i(T, \omega) \in \alpha_j^i \lambda_i  \partial_x g (x_i(T, \omega); \omega) + \alpha_j^i(1-\lambda_i) \partial d_{C(\omega)} (x_i(T, \omega))  	 \ ;  $
	  	\item $\sum_{\omega \in \Omega^i}  p_i(t,\omega) \cdot [f(t, x_i(t,\omega),u,\omega) -  f(t,x_i(t,\omega),\bar u(t),\omega) ] \leq  \rho_i'$ \quad for all $t \in A_{\rho_i'}$\ and for any $u \in U(t) \ .$
	  \end{enumerate}
  \noindent
  Following the idea of Proposition \ref{finitenco}, and dividing  each term of the family of the costate arcs  across by the corresponding coefficient $\alpha_j^i(>0)$ (without relabelling), we obtain that for each $i$ large enough and $\mu_i-$a.e. $\omega \in \Omega$,
  \begin{enumerate}[label=(a\arabic*)$'$, ref=(a\arabic*)$'$]
  	\item  \label{a0} $ p_i(.,\omega) \neq 0$ \ ;
  	\item \label{a1} $ - \dot { p}_i (t, \omega) \in \textrm{co } \partial_x [ p_i(t,\omega) \cdot f(t, x_i(t,\omega), \bar u(t), \omega)]$ \quad for all $t \in A_{\rho_i'}$ ;
  	\item \label{a2} $-  p_i(T, \omega) \in \lambda_i  \partial_x g (x_i(T, \omega); \omega) + (1-\lambda_i) \partial d_{C(\omega)} (x_i(T, \omega))  	 \ ;  $
  	\item \label{a3} $\int_{ \Omega}  p_i(t,\omega) \cdot [f(t, x_i(t,\omega),u,\omega) -  f(t,x_i(t,\omega),\bar u(t),\omega) ] \ d\mu_i(\omega) \leq  \rho_i'$ \quad for all $t \in A_{\rho_i'}$\ and for any $u \in U(t) \ .$
  \end{enumerate}

\vskip2ex
\noindent{\bf 3.} We derive now consequences of the limit-taking for conditions \ref{a0}-\ref{a2} of the previous step. 
Recall that from Lemma \ref{lemma_relation_sequence_dirac_measure_with_given_probability}, we have a countable dense subset $\widehat \Omega$ of $\Omega$, such that
$ \widehat \Omega \ = \  \bigcup_{i \ge 1} \Omega^i \ ,     $
where $\Omega^i = \{ \omega_j^i \ : \ j =0,\ldots,N_i   \}$ provides an increasing sequence of finite subsets of $\Omega$: $\Omega^1 \subset \ldots \subset \Omega^i \subset \Omega^{i+1} \subset \ldots$. Since $\widehat \Omega$ is a countable set, we can write it as the collection of the elements of a sequence $\{ \omega_k \}_{k \ge 1}$ such that
 \[ \widehat \Omega = \{ \omega_k  \}_{k \ge 1} . \]
 \noindent
Fix $i\in \N$. When we take $\omega_k \in \widehat{\Omega}$, two possible cases may occur: either $\omega_k \in \Omega^i$ for the fixed $i \in \N$; or $\omega_k \in \widehat \Omega \setminus \Omega^i$. In the first case, it means that there exists $j \in \{0,\ldots,N_i\}$ such that $\omega_k = \omega_j^i$ and the corresponding adjoint arc $p_i(.,\omega_j^i)$ satisfies conditions \ref{a0}-\ref{a3}. So, we can define the arc $p_i(.,\omega_k)$ as follows:
\[ p_i(.,\omega_k) := \begin{cases} p_i(.,\omega_j^i) \quad & \text{if } \omega_k \in \Omega^i \; (\text{and } \omega_j^i = \omega_k) \\ 0 \quad & \text{if }\omega_k \in \widehat \Omega \setminus \Omega^i . \end{cases}   \]
Therefore, by iterating on $i$, associated with each $\omega_k \in \widehat{\Omega}$, we can construct a sequence of families of arcs $\{ p_i(.,\omega_k): \omega_k \in \widehat \Omega \}_{i \ge 1}$.
Observe that there exists always $i_k \in \N$ such that, for all $i \ge i_k$, $p_i(.,\omega_k)$ is an adjoint arc for which \ref{a0}-\ref{a3} hold true. 
%
From \ref{a2} and (A4)$'$ it immediately follows that the sequence $\{ p_i(T,\omega_k) \}$ is uniformly bounded by $k_g+1$. On the other hand \ref{a1} and (A3)$'$ imply that $\{ \dot p_i(.,\omega_k) \}$ are uniformly integrably bounded.
Then, the hypotheses are satisfied under which the Compactness Theorem \cite[Theorem 2.5.3]{vinter2010optimal} is applicable to
\[  -\dot p_i(t,\omega_k) \in \text{co }\partial_x [p_i(t,\omega_k) \cdot f(t,x_i(t,\omega_k),\bar u(t),\omega_k) ] \quad \text{for all } t \in A_{\rho_i'} .   \]
We conclude that, along some subsequence (we do not relabel),
\begin{equation} \label{equation: convergence property 1. compactness theorem}    p_{i}(.,\omega_k) \xrightarrow[i]{} \widehat p(.,\omega_k)  \;\; \text{uniformly} \qquad \text{and} \qquad \dot p_{i}(t,\omega_k) \rightharpoonup \dot {\widehat p} (t,\omega_k) \quad \text{weakly in } L^1 \end{equation}
for some $\widehat p(.,\omega_k) \in W^{1,1}$ which satisfies (for the fixed $k$)
\[  -\dot{\widehat{p}}(t,\omega_k) \in \text{co }\partial_x [\widehat{p}(t,\omega_k) \cdot f(t,\bar x(t,\omega_k),\bar u(t),\omega_k)] \quad \text{a.e. } t\in [0,T].    \]
We can also take the subsequence in such a manner that $\{\lambda_i \}$ converges to some $\lambda \in [0,1]$.
Moreover, from the closure of the graph of the limiting subdifferential and the normal cone (seen as multifunctions), we have that
\[  -  \widehat{p}(T, \omega_k) \in \lambda \partial_x g (\bar x(T,\omega_k); \omega_k) + (1-\lambda)  \partial d_{C(\omega_k)} (\bar x(T,\omega_k)) . \]
But $\widehat \Omega = \{ \omega_k  \}_k$ is a countable set. Then, we can repeat the similar analysis for each $\omega_k \in \widehat \Omega$, taking into account the subsequence obtained for the previous element $\omega_{k-1}$. As a consequence, we have a collection of subsequences $\{ \widetilde p_i(.,\omega)  \}$ verifying the convergence properties (\ref{equation: convergence property 1. compactness theorem}) to a collection of adjoint arcs $\{ \widetilde p(.,\omega)  \}$ which satisfies, for all $\omega \in \widehat \Omega$
\begin{equation} \label{equation: adjoint system proof general case}  -\dot{\widetilde{p}}(t,\omega) \in \text{co }\partial_x [\widetilde{p}(t,\omega) \cdot f(t,\bar x(t,\omega),\bar u(t),\omega)] \quad \text{ a.e. } t\in [0,T]   \end{equation} and
\begin{equation*}  -  \widetilde{p}(T, \omega) \in \lambda \partial_x g (\bar x(T,\omega); \omega) + (1- \lambda) \partial d_{C(\omega)} (\bar x(T,\omega))   . \end{equation*}
Furthermore, since for all $i$, $\widetilde p_i(.,.)$ is $\mathcal{L} \times \mathcal{B}_{\widehat{\Omega}}$ measurable, we obtain that its limit $\widetilde p(.,.)$ is also $\mathcal{L} \times \mathcal{B}_{\widehat \Omega}$ measurable. 
The final step is represented by the extension of $\widetilde p(.,.)$ to a $\mathcal{L} \times \mathcal{B}_{\Omega}$ measurable function $p(.,.)$ on $[0,T] \times \Omega$ which satisfies conditions (\ref{equation: adjoint system proof general case}) and (\ref{equation: transversality condition proof general case}) below when restricted to  $\widehat \Omega$. 
This can be done as follows. Writing explicitly the coordinates of $\widetilde p(.,.)=(\widetilde p ^1(.,.), \dots, \widetilde p ^n(.,.))$, for each $j=1,\dots, n$, we have the decomposition into the positive and negative parts: $\widetilde p ^j=\widetilde p ^{j+}-\widetilde p ^{j-}$. Consider a sequence of simple functions $\widetilde \phi_k(.,.)$ (for $\mathcal{L} \times \mathcal{B}_{\widehat \Omega}$) which approximates from below $\widetilde p ^{j+}(.,.)$: $0\le \widetilde \phi_k \uparrow \widetilde p ^{j+}$. Let $\phi_k(.,.)$ be the simple function which provides an extension of $\widetilde \phi_k(.,.)$ to $\mathcal{L} \times \mathcal{B}_{\Omega}$. Then, define 
\begin{equation}\label{limit}
p^{j+}(t,\omega) := \begin{cases} \lim_k \phi_k(t,\omega) \quad & \text{if the limit exists and is finite} \\ 
0 \quad & \text{otherwise }. 
\end{cases}   
\end{equation} 
Then, we obtain the desired extension setting $p^j=p^{j+} - p^{j-}$ and $p(.,.)=(p^1(.,.), \dots, p^n(.,.))$.
Clearly we have the following transversality condition:
\begin{equation}\label{equation: transversality condition proof general case}  -  {p}(T, \omega) \in \lambda \partial_x g (\bar x(T,\omega); \omega) +N^1_{C(\omega)} (\bar x(T,\omega)), \quad \mbox{for all } \; \omega \in \widehat{\Omega}  . \end{equation}


\noindent
Finally, we derive a non-triviality condition for $\{p(.,\omega) \; : \; \omega \in \widehat{\Omega}\}$. This is immediate if the $\lambda=\lim \lambda_i>0$, so we continue examining the case in which $\lambda=0$. 
Choose $i_0 \in \N$ such that for all $i \ge i_0$, $(k_g +1) \lambda_i < \frac{1}{2}$.
In particular, for all $i \ge i_0$, from the Max Rule we have $1-\lambda_i>0$, and using the fact that $\widetilde{J}_i(u_i) >0$, it follows that 
$$
\int_{\Omega^i}d_{C(\omega)} (x_i(T,\omega))\ d\mu_i(\omega) = \sum_{j=0}^{N_i} \alpha_j^i d_{C(\omega_j^i)} (x_i(T,\omega_j^i)) > 0.
$$ 
Then there exists $j_i \in \{0,1, \dots, N_i\}$  and $\nu\in \R^n$ such that $|\nu|=1$ and 
 \[  -  p_i(T, \omega_{j_{i}}^i) \in  \lambda_i  \partial_x g (x_i(T, \omega_{j_{i}}^i); \omega_{j_{i}}^i) + (1-\lambda_i)\nu \ .  \]
Recalling that $k_g>0$ is the Lipschitz constant of $g(.,\omega)$, we have
\[  |p_i(T,\omega_{j_{i}}^i)|   \ge  -\lambda_i k_g +  (1-\lambda_i)  .\]
And from the choice of $i_0 \in \N$, we obtain that
\[ |p_i(T,\omega_{j_{i}}^i)| \ge \frac{1}{2} , \]
and so
\begin{equation}  \label{new non-degeneracy}  \sum_{j=1}^{N_i} \| {p}_i(.,\omega_j^i) \|_{L^\infty} \ge  \| p_i(.,\omega_{j_{i}}^i)\|_{L^\infty}   \ge |p_i(T,\omega_{j_{i}}^i)|  \ge \frac{1}{2}  .  \end{equation}
We deduce that
\begin{equation*} \sum_{\omega \in \widehat \Omega} \max_{t \in [0,T]} |p(t,\omega)| \ge \frac{1}{2}.  \end{equation*}
In any case, we obtain the  non-triviality condition 
\begin{equation} \label{non-triviality non zero} 
(\lambda, \{p(.,\omega) \; : \; \omega \in \widehat{\Omega}\}) \neq (0,0) \ .
\end{equation}


\vskip2ex
\noindent{\bf 4.} In the last part of the proof, we want to use also the information contained in the maximality condition \ref{a3} (or in its alternative version (\ref{maximization_condition_with_integral_inequality})) as $i \rightarrow \infty$. This task requires to use Castaing's Representation Theorem (cf. \cite[Theorem III.7]{castaing2006convex}, 
the Aumann's Measurable Selection Theorem (cf. \cite[Theorem III.22]{castaing2006convex}), and Lemma \ref{weakconvergence} which has a central role for the limit-taking of all the necessary conditions obtained in Step 2 at the same time. 
Write
\[ F(t,\omega) := f(t,\bar x (t,\omega),U(t),\omega)  \ .  \] Owing to assumption (\ref{A1_nco_general_case}) and the Lipschitz continuity of $f(t,.,u,\omega)$, we obtain that $(t,\omega) \leadsto F(t,\omega)$ is a $\mathcal{L} \times \mathcal{B}_\Omega$ measurable with closed values. Using the Castaing's Representation Theorem, we know that there exists a countable family of $\mathcal{L} \times \mathcal{B}_\Omega$ measurable functions $\{ f_j(t,\omega)  \}_{j \ge 0}$, such that
\[   F(t,\omega) = \overline{\bigcup\limits_{j \ge 0}\{ f_j(t,\omega)   \} }      \quad \text{for all } (t,\omega) \in E   \ ,     \]in which $E \subset [0,T] \times \Omega$ is a set of full-measure. We can also assume that $f_0(t,\omega) = f(t, \bar x (t,\omega), \bar u (t), \omega)$. For all $j \ge 1$, define the multifunction

\[ \widetilde{U}_j(t,\omega) := \begin{cases}
\bar u(t)   & \text{if } (t,\omega) \notin E \\ \{ u \in U(t) \ : \  f_j(t,\omega) = f(t, \bar x (t,\omega),  u, \omega)  \} & \text{if } (t,\omega) \in E \ .
\end{cases}             \]
The graph of $\widetilde{U}_j(.,.)$ is a $\mathcal{L} \times \mathcal{B}_\Omega \times \mathcal{B}^m$ measurable set. Indeed, we have
\begin{align*}
\text{Gr }\widetilde{U}_j(.,.) =  \{ ((t, &\omega),u)  \ : \ u \in U(t) ,   \ (t,\omega) \in E \ \text{ and } \ f(t, \bar x (t,\omega),  u, \omega) - f_j(t,\omega)   =0 
\} \\ & \bigcup \  \{ ((t,\omega),u) \ : \ (t,\omega) \notin E , \ u = \bar u(t)   \} \ ,\end{align*} which is the union of two $\mathcal{L} \times \mathcal{B}_\Omega \times \mathcal{B}^m$ measurable sets. Now invoking Aumann's Measurable Selection Theorem, we deduce that $\widetilde{U}_j(.,.)$ has a measurable selection $v_j(t,\omega) \in \widetilde{U}_j(t,\omega)$.\\
Let now $\mathcal{D}$ be a countable and dense subset of $[0,T]$. Consider the sequence of intervals $\{ [s_i , t_i]  \}_{i \ge 1}$  having extrema in $\mathcal{D} : \ \bigcup\limits_{i \ge 1} \{ s_i ,t_i \} = \mathcal{D}$. We construct now a further countable family of controls $\{ \widetilde v_{j,i} (t,\omega)  \}_{j \ge 1, \ i \ge 1}$ as follows
\begin{equation} \label{system: construction of another family of controls }   \widetilde v_{j,i} (t,\omega) : = \begin{cases}
 v_j(t,\omega)  & \quad \text{on } [s_i,t_i] \times \Omega \\ \bar u(t ) & \quad \text{on } ([0,T] \setminus [s_i,t_i]) \times \Omega \ . 
\end{cases}            \end{equation}
Writing $\{ \widetilde{u}_k(t,\omega)  \}_{k \ge 0} = \{ \widetilde{v}_{j,i}(t,\omega)  \}_{j \ge 1, i \ge 1} \cup \{ \bar{u}(.)  \}$, in such a manner that (up to a reordering) $\widetilde{u}_0(.,\omega) = \bar{u}(.)$, we obtain
\begin{equation} \label{form of set of velocities}  F(t,\omega) = \overline{\bigcup\limits_{k \ge 0}\{f(t, \bar x (t,\omega), \widetilde{u}_k(t,\omega), \omega)    \} }      \quad \text{for all } (t,   \omega) \in E    \ .     \end{equation}

\noindent
Following an effective technique proposed by Vinter \cite{vinter2005minimax}, for a fixed integer $K$, we introduce the ope\-rators $\Psi_k(.,.)$ and $\Psi^i_k(.,.)$ on $W^{1,1}([0,T], \R^n)\times \Omega$ (linear with respect to their first variable): for $k=1, \ldots, K$, we set
\[  \Psi_k(p(.),\omega) := \int_{0}^{T} p(t) \cdot [f(t, \bar x(t,\omega),\widetilde u_k(t,\omega),\omega)   -  f(t, \bar x(t,\omega),\bar u(t),\omega)]  \ dt  \ , \] and, for all integers $i\ge 1$,
\[  \Psi^i_k(p(.),\omega) := \int_{0}^{T} p(t) \cdot [f(t, x_i(t,\omega),\widetilde u_k(t,\omega),\omega)   -  f(t,  x_i(t,\omega), u_i(t),\omega)] \  dt  \ . \]
Define also the subsets $D_i$, for all $i\ge 1$, and $D$ of $\Omega \times \R^K$ as follows: 
\begin{align*} D_{i}:= \{ (\omega,\xi) & \in \Omega \times \R^K \ | \ \omega \in \Omega \; \textrm{ and }  \xi = \big( \Psi^i_k(p(.,\omega), \omega) \big) _{k=1,\ldots,K}  \text{for some }\mathcal{L}\times\mathcal{B}_\Omega \text{ measurable function } \\ & p: [0,T] \times \Omega \rightarrow \R^n \;  \textrm{ such that }  \  p(.,\omega) \in \mathcal{P}_{i}(\omega) \ \text{ for all } \omega \in \Omega^i   \}, \end{align*}
where $\{\Omega^i\}$ is the increasing sequence of (finite) subsets introduced in Step 3 (cf. Lemma \ref{lemma_relation_sequence_dirac_measure_with_given_probability}), and 
\begin{align*}
\mathcal{P}_{i}(\omega) & := \Bigg\{ \nonumber q(., \omega) \in W^{1,1} \ : \  
\ , \ - \dot q(t, \omega) \in \textrm{co } \partial_x [q(t,\omega) \cdot f(t, \bar x (t, \omega) + \epsilon'_i \B, \bar u(t), \omega)] \\ & \nonumber  \text{on a set } A_i \text{ such that } \text{meas } ([0,T] \setminus A_i) \le \rho'_i \ ,  \  \textrm{and there exists } \lambda_i \in [0,1] \text{  such that }\ \\ & (\lambda_i, \{q(.,\omega) \; : \; \omega \in \Omega^i \}) \neq (0,0)  \ \text{ and }\  - q(T, \omega)  \in \bigcup\limits_{ x \in \bar x (T, \omega) + \epsilon_i' \B} \; \; \lambda_i \partial_x g (x,\omega) + N^1_{C(\omega)}(x)  \Bigg\} \ ,
\end{align*}
in which $\epsilon_i' : = \beta\rho_i'$. 
The set $D$ is written
\begin{align*} D := \{ (\omega,\xi) & \in \Omega \times \R^K \ | \ \omega \in \Omega \; \textrm{ and }  \xi = \big( \Psi_k(p(.,\omega), \omega) \big) _{k=1,\ldots,K} \textrm{ for some } \mathcal{L}\times\mathcal{B}_\Omega \text{ measurable function } \\ & p: [0,T] \times \Omega \rightarrow \R^n \ \text{such that }  \ \ p(.,\omega) \in \text{co }\mathcal{P}(\omega) \ \text{ for all }\omega \in \widehat{\Omega} \} 
\end{align*}
where $\widehat{\Omega}$ is the countable dense subset of $\Omega$ ($=\mbox{supp} (\mu)$ in our assumptions) provided by Lemma \ref{lemma_relation_sequence_dirac_measure_with_given_probability} and

\begin{align*} \mathcal{P}(\omega) :=  & \Bigg\{ q(.,\omega) \in  W^{1,1}( [0,T], \R^n) \ : \  \text{for some } \lambda \in [0,1], \text{ we have }  
\ , \\ &  (\lambda, \{q(.,\omega) \; : \; \omega \in \widehat{\Omega}\}) \neq (0,0)  \ , \   - \dot q(t,\omega) \in \text{co }\partial_x [q(t,\omega) \cdot f(t,\bar x(t,\omega),\bar u(t),\omega)] \text{ a.e. } t \in [0,T] \ ,\\ &
-q(T,\omega) \in \lambda \partial_x g(\bar x(T,\omega);\omega) + N^1_{C(\omega)}(\bar x(T,\omega)) \Bigg\}. \end{align*}
\noindent
Now, we define the multifunctions $D_i(.)$, for $i=1,2,\ldots$, and $D(.)$ on $\Omega$, taking values in the subsets of $\R^K$ as follow:
\[ D_{i}(\omega) := \{ (\xi_1,\ldots,\xi_K) \in \R^K \ : \ (\omega,\xi) \in D_i  \}  \qquad \text{and} \qquad D(\omega) := \{ (\xi_1,\ldots,\xi_K) \in \R^K \ : \ (\omega,\xi) \in D  \} . \]
\noindent
The multifunctions $\omega \leadsto D(\omega)$ and $\omega \leadsto  D_i(\omega)$, for all $i$, are uniformly bounded. The necessary optimality conditions \ref{a0}-\ref{a2} corresponding to the auxiliary problem (\ref{auxiliaryprob}) of Step 2 guarantee that the set $D_{i}(\omega)$ is non-empty : indeed there exist $\mathcal{L}\times\mathcal{B}_\Omega$ measurable functions $p_i:[0,T] \times \Omega \rightarrow \R^n$ such that $p_i(.,\omega) \in \mathcal{P}_{i}(\omega)$ $\mu_i-$a.e. $\omega \in \Omega$ and so
\[  \big( \Psi^i_k(p_i(.,\omega), \omega) \big) _{k=1,\ldots,K} \in D_i(\omega) \quad  \mu_i- \textrm{a.e. } \omega \in \Omega \ .  \]
 Moreover, the linearity of the operator $\Psi_k$ with respect to the first variable $p$ and the convexity of the set $\text{co }\mathcal{P}(\omega)$ guarantee the convexity of the set $D(\omega)$ for each $\omega \in \text{dom }D(.)$. It follows that hypotheses (\ref{item: i of proposition weak convergence})-(\ref{item: iii of proposition weak convergence}) of Lemma \ref{weakconvergence} are satisfied. We claim that
\[ \limsup\limits_{i \rightarrow \infty} D_i \subset D \ . \]
Indeed, take any $(\omega, \xi) \in \limsup\limits_{i \rightarrow \infty} D_i$. 
From the definition of the limsup in the Kuratowski sense, there exists a subsequence $i_h \rightarrow \infty$ and $(\omega_{i_{h}},\xi_{i_h}) \in D_{i_h}$ such that
\[  \lim\limits_{i_h \rightarrow \infty}(\omega_{i_h}, \xi_{i_h}) = (\omega, \xi) \]
We shall show that $(\omega,\xi) \in D$. Since $(\omega_{i_{h}},\xi_{i_h}) \in D_{i_h}$, there exists a sequence of $\mathcal{L}\times\mathcal{B}_\Omega$ measurable functions $p_{i_h} :[0,T] \times \Omega \rightarrow \R^n$ such that $p_{i_h}(.,\omega) \in \mathcal{P}_{i_h}(\omega)$ for all $\omega \in \Omega^{i_h}$. From the analysis of Step 3, we have established the existence of a map $p$ on $[0,T]\times\Omega$ which is $\mathcal{L}\times\mathcal{B}_\Omega$ measurable, verifying conditions (\ref{equation: adjoint system proof general case}), (\ref{equation: transversality condition proof general case}), and  (\ref{non-triviality non zero}) for all $\omega \in \widehat{\Omega}$. 
Moreover, the uniform convergence of $\{p_{i_h}(.,\omega) \ : \ \omega \in \widehat{\Omega}\}$, Lemma \ref{lemma_continuity_omega_control} and assumption \ref{A2'_nco_general_case} guarantee that, for $k=1,\ldots,K$ and for all $\omega \in \widehat \Omega$,
 \begin{align*}
 \int_0^T p_{i_h}(t,\omega) \cdot \left[ f(t,{x_{i_h}}(t,\omega),\widetilde{u}_k(t,\omega),\omega) - f(t,{x_{i_h}}(t,\omega),u_i(t),\omega)  \right] \ dt  
 \end{align*} converges, as $i_h \rightarrow \infty$, to
\begin{align*}
 \int_0^T p(t,\omega) \cdot \left[ f(t,\bar{x}(t,\omega),\widetilde{u}_k(t,\omega),\omega) - f(t,\bar{x}(t,\omega),\bar{u}(t),\omega)  \right] \ dt  \ .
 \end{align*}
 Therefore, $(\omega,\xi) \in D$ and the claim is confirmed. Consequently, all required hypotheses of Lemma \ref{weakconvergence} are satisfied for $\gamma_i(\omega)=(\gamma_{i,1}(\omega),\ldots,\gamma_{i,K}(\omega))$ where for $k=1,\ldots,K$, \[ \gamma_{i,k}(\omega)= \int_{0}^{T} p_i(t,\omega) \cdot [f(t, {x_i}(t,\omega),\widetilde u_k(t,\omega),\omega)   -  f(t, {x_i}(t,\omega), u_i(t),\omega)]  \ dt \] which is $\mu_i-$measurable. Defining, for each $i$, the vector-valued measure $\eta_i := \gamma_i \mu_i$, and applying Lemma \ref{weakconvergence}, we obtain, along a subsequence (we do not relabel)
 $ \eta_i  \stackrel{*}\rightharpoonup \eta   $ where $\eta$ is a vector-valued Borel measure on $\Omega$ such that
 $ d\eta (\omega) = \gamma(\omega) \ d\mu (\omega)$, for some Borel measurable function $\gamma : \Omega \rightarrow \R^K$ satisfying
 \[    \gamma(\omega)  \in D(\omega) \quad \mu-\text{a.e. } \omega \in \Omega \ .   \]
 In addition, from the definition of the set $D$ (associated with each $K\in \N$), there exists a $\mathcal{L}\times\mathcal{B}_\Omega$ measurable function $p_K:[0,T]\times\Omega \rightarrow \R^n$ such that 
$p_K(.,\omega) \in \text{co }\mathcal{P}(\omega)$ for all $\omega \in \widehat{\Omega}$, and $\gamma (\omega) : = \big(   \Psi_k(p_K(.,\omega),\omega)    \big)_{k=1,\ldots,K}$ verifying
 \[  \int_{\Omega} \gamma_i(\omega) \ d\mu_i(\omega) \xrightarrow[i \rightarrow \infty ]{}  \int_{\Omega} \gamma(\omega) \ d\mu(\omega)   \ .  \]
 In other terms, for each $k=1,\ldots,K$
 
 \begin{align} \label{convergence_due_to_limit_taking_result}
 &  \nonumber\int_{\Omega} \int_{0}^{T} p_i(t,\omega) \cdot [f(t, {x_i}(t,\omega),\widetilde u_k(t,\omega),\omega)   -  f(t, {x_i}(t,\omega), u_i(t),\omega)] \ dt d\mu_i(\omega) \\ & \xrightarrow[i \rightarrow  \infty]{} \int_{\Omega} \int_{0}^{T} p_K(t,\omega) \cdot [f(t, {\bar x}(t,\omega),\widetilde u_k(t,\omega),\omega)   -  f(t, {\bar x}(t,\omega), \bar u(t),\omega)] \ dt d\mu(\omega) \ .
 \end{align}
 The maximality condition \ref{a3} of Step 2, after inserting $u=\widetilde u_k(t,\omega)$, gives
 \begin{equation} \label{equation: integral and integrand proof}\int_\Omega  p_i(t,\omega) \cdot [f(t, x_i(t,\omega),\widetilde u_k(t,\omega),\omega) -  f(t,x_i(t,\omega), u_i(t),\omega) ] \ d\mu_i(\omega) \leq \rho_i' \  \text{ a.e. }t \in [0,T] . \end{equation} Since in (\ref{equation: integral and integrand proof}) the integrand function is $\mathcal{L} \times \mathcal{B}_\Omega-$measurable, and the integral function is $\mathcal{L}-$measurable, making use of Fubini-Tonelli, we obtain
 \[\int_\Omega \int_{0}^T p_i(t,\omega) \cdot [f(t, x_i(t,\omega),\widetilde u_k(t,\omega),\omega) -  f(t,x_i(t,\omega), u_i(t),\omega) ] \ dt d\mu_i(\omega) \leq  \rho'_i T \ . \]
 Therefore, letting $i \to \infty$ and invoking (\ref{convergence_due_to_limit_taking_result}), we have that
 \begin{equation} \label{maxqN}   \int_{\Omega} \int_{0}^{T} p_K(t,\omega) \cdot [f(t, {\bar x}(t,\omega),\widetilde u_k(t,\omega),\omega)   -  f(t, {\bar x}(t,\omega), \bar u(t),\omega)]  \ dt d\mu(\omega)  \le 0 \ .    \end{equation}

\noindent
For each $K \in \N$, the map $\omega \rightarrow p_K(.,\omega)$ can be interpreted as a $\mathcal{B}_\Omega-$measurable element of the $\mu-$a.e. equivalence class in the Hilbert space
\[ \mathcal{H} := L^2_{\mu}(\Omega, L^2( [0,T]; \R^n)) \] endowed with the inner product
\[ \big<p,p'\big>_\mu := \int_\Omega \int_{0}^{T} p(t,\omega) \cdot p' (t,\omega) \ d t d \mu(\omega) \ .  \]
\noindent
Now consider $\widehat{\mathcal{P}}$ to be the set of $\mathcal{L}\times\mathcal{B}_{\Omega}$ measurable functions $\widehat q$ of $\mathcal{H}$ defined on $[0,T] \times \Omega$ such that $\widehat q (.,\omega) \in \text{co } \mathcal{P}(\omega)$ for all $\omega \in \widehat{\Omega}$: 
\[  \widehat{\mathcal{P}} := \{  \widehat q \in {\mathcal{H}} \ | \ \widehat q (.,\omega) \in  \text{co }\mathcal{P}(\omega) \quad \text{for all } \omega \in \widehat{\Omega} \}    \ .  \]
Note that $ \widehat{\mathcal{P}}$ is nonempty since $p_K(.,\omega) \in  \text{co }\mathcal{P}(\omega)$ for all $\omega \in \widehat{\Omega}$. Moreover, it is a straightforward task to prove that $\widehat{\mathcal{P}}$ is a closed and convex subset in $\mathcal{H}$ (owing to the convexity and the closure of the set $\text{co }\mathcal{P}(\omega)$ for all $\omega \in \widehat{\Omega}$). Therefore, $\widehat{\mathcal{P}}$ is weakly closed, as well. 
Moreover, the sequence $\{ \omega \rightarrow p_K(.,\omega)   \}_{K=1}^{\infty}$ is (uniformly) bounded, w.r.t. the norm induced by $\big< ., . \big>_\mu$ because it belongs to the bounded set $\text{co }\mathcal{P}(\omega)$ for all $\omega \in \widehat \Omega$. By subsequence extraction (without relabelling), there exists a weakly convergent subsequence to $\{ \omega \rightarrow p(.,\omega)   \}$ for some $p \in \widehat{\mathcal{P}}$. The weak convergence $p_K \rightharpoonup p$ in the Hilbert space $(\mathcal{H}, \big< .,. \big>_\mu)$, employed in inequality (\ref{maxqN}), implies that
\begin{equation} \label{integral_weak_convergence} \int_{\Omega} \int_{0}^{T} p(t,\omega) \cdot [f(t, \bar x(t,\omega),\widetilde u_k(t,\omega),\omega)  -  f(t, \bar x(t,\omega),\bar u(t),\omega)]  \ dt d\mu(\omega) \leq 0 \ . \end{equation}


\noindent We observe that condition (\ref{form of set of velocities}) yields the following inclusion for all $t\in \mathcal{S}$
\begin{align} \label{inclusion_castaing} 
\int_{\Omega} p(t,\omega) \cdot f(t, & \bar x(t,\omega) ,U(t),\omega) \ d \mu (\omega) ~\ \mathlarger{\mathlarger{{\subset}}} ~\
\overline{\bigcup\limits_{k \ge 0} \bigg\{\int_{\Omega} p(t,\omega) \cdot f(t,\bar x(t,\omega), \widetilde{u}_{k}(t,\omega),\omega) \   d \mu (\omega)\bigg\} } \ ,
\end{align}
where $\mathcal{S}$ is a set of full measure in $[0,T]$. 
Define now the set $\mathbcal S' \subset \mathbcal S$, still of full measure in $[0,T]$, containing the Lebesgue points for the map $\Gamma:[0,T] \rightarrow \R$ defined as
\[ s \mapsto \Gamma(s) := \int_{\Omega} p(s,\omega) \cdot [f(s,\bar x (s, \omega), \widetilde {u}_k(s,\omega),\omega)  - f(s,\bar x (s, \omega), \bar u(s),\omega)] \  d \mu(\omega)  \]for all $k$.
Take any $t \in \mathcal S'$ and $u \in U(t)$. Owing to (\ref{inclusion_castaing}), there exists a subsequence $\{k_\ell\}_\ell$ such that
\begin{align*}
 \int_{\Omega} p(t,\omega) \cdot  f(t,\bar x (t, \omega), u,\omega) \ d \mu(\omega)  = \lim\limits_{\ell}  \int_{\Omega} p(t,\omega) \cdot f(t,\bar x (t, \omega), \widetilde {u}_{k_\ell}(t,\omega),\omega) \ d \mu(\omega) \ . \end{align*}
In other words, for a sequence $\beta_\ell \downarrow 0$ (possibly taking a subsequence of $\widetilde{u}_{k_\ell}$), we have
\begin{align} \label{inequality_adjoint_arc_inclusion_castaing} \bigg| \int_{\Omega} p(t,\omega) \cdot  f(t,\bar x (t, \omega), u,\omega) \  d \mu(\omega) -  \int_{\Omega} p(t,\omega) \cdot f(t,\bar x (t, \omega), \widetilde {u}_{k_\ell}(t,\omega),\omega) \  d \mu(\omega) \bigg| \le \beta_\ell \ . \end{align}
For the Lebesgue point $t \in  \mathbcal S'$, we can also consider a sequence of intervals $\{ [s_i,t_i]  \}_{i \ge 1}$, having extrema in a countable dense set $\mathcal{D}$ of $[0,T]$ (in the sense of (\ref{system: construction of another family of controls })) and such that $s_i \uparrow t$ and $t_i \downarrow t$. 
Recalling the definition (\ref{system: construction of another family of controls }) of $\widetilde v_{j,i}$ and replacing in (\ref{integral_weak_convergence}) $\widetilde u_k$ by $v_{j}(t,\omega)$ on $[s_i,t_i] \times \Omega$, and by $\bar u(t)$ on $([0,T]\setminus [s_i,t_i]) \times \Omega$, using Fubini-Tonelli (since the integrand is $\mathcal{L} \times \mathcal{B}_\Omega-$measurable) and dividing across by $|t_i-s_i|$, we obtain
\begin{equation} \label{fubini_tonelli_inequality} \frac{1}{|t_i-s_i|} \int_{s_i}^{t_i} \int_\Omega  p(s,\omega) \cdot  [ f(s,\bar x (s,\omega), {v}_j(s,\omega),\omega) - f(s,\bar x (s,\omega),\bar u(s),\omega)] \ d\mu(\omega) ds\leq 0 \ .    \end{equation}
Since $t$ is a Lebesgue point for the map $\Gamma$, we deduce
\begin{align} \label{lebesgue_point_definition} \nonumber & \int_\Omega  p(t,\omega) \cdot  [f(t,\bar x (t,\omega),\widetilde{u}_k(t,\omega),\omega)  - f(t,\bar x (t,\omega),\bar u(t),\omega)   ] \ d\mu(\omega) \\  & = \lim\limits_{i}  \frac{1}{|t_i-s_i|} \int_{s_i}^{t_i} \int_\Omega  p(s,\omega) \cdot  [ f(s,\bar x (s,\omega), {v}_j(s,\omega),\omega) - f(s,\bar x (s,\omega),\bar u(s),\omega)   ] \ d\mu(\omega) ds. 
\end{align}
Therefore, owing to (\ref{inequality_adjoint_arc_inclusion_castaing})-(\ref{lebesgue_point_definition}), we have
\[ \int_{\Omega} p(t,\omega) \cdot [f(t,\bar x (t,\omega),u,\omega) - f(t,\bar x (t,\omega),\bar u(t),\omega) ] \ d\mu(\omega) \leq \beta_\ell + 0 \ ,   \] for any $\beta_\ell \downarrow 0$ and any $u\in U(t)$.
We conclude that
\[ \int_{\Omega} p(t,\omega) \cdot [f(t,\bar x (t,\omega),u,\omega) - f(t,\bar x (t,\omega),\bar u(t),\omega) ] \  d\mu(\omega) \leq 0    \]
for any $u\in U(t)$ and for all $t \in \mathcal S'$, a set of full measure in $[0,T]$. Therefore, now all the assertions stated in Theorem \ref{theorem_nco_general_case_probability_space} are confirmed (included the maximality condition (\ref{item: maximality condition general case})), which completes the proof. 
\qed

\ \\
\noindent
{\bf Proof of Theorem \ref{atomic}}.
A purely atomic measure has necessarily at most a countable support. We can therefore choose $\widehat{\Omega}$ in such a manner that $\widehat{\Omega}= \text{supp}(\mu)$.
The properties (\ref{item: absolute continuity of adjoint arcs atomic case}) and (\ref{item: adjoint system+nontriviality+transversality atomic})  follow immediately considering Steps 1, 2 and 3 of Theorem \ref{theorem_nco_general_case_probability_space} proof and the obtained costate arc $p(.,.)$.
On the other hand, the maximality condition (\ref{item: maximality condition atomic case}) can be deduced by contradiction, avoiding any use of the technical procedure of Step 4 of Theorem \ref{theorem_nco_general_case_probability_space} proof which requires the construction of appropriate multifunctions and the use of selection theorems. 
We provide here the details of this `new step 4' which allows to obtain (ii).

%
%

\noindent
Consider the function
\[ (t,u) \to \Psi(t,u) := \sum_{k \ge 0}  \mu(\omega_k) p(t,\omega_k) \cdot \big[ f(t,\bar x(t,\omega_k),u,\omega_k) - f(t,\bar x(t,\omega_k),\bar u(t),\omega_k)  \big] .  \]
Using a standard argument, one can easily show that
\[ (t,u) \to \Psi(t,u) \text{ is } \mathcal{L}\times \mathcal{B}^m-\text{measurable}.   \]
Therefore, setting, for $j\in \N$
\[ E_j := \bigg\{ (t,u) \ : \ \Psi(t,u) \ge \frac{1}{j}   \bigg\} \subset [0,T] \times \R^m ,  \]
we have that $E_j$ is a $\mathcal{L}\times \mathcal{B}^m-$measurable set. Define
\[  B_j := \bigg\{  t \ : \ (t,u) \in E_j \cap \text{Gr } U(.)  \bigg\}.  \]
Then $\{B_j\}_{j \ge 1}$ is an increasing sequence of $\mathcal{L}-$measurable sets. 
%
Consider the following $\mathcal{L}\times \mathcal{B}^m-$measurable set $E$
\[  E := \Psi^{-1} (]0,+\infty[) \cap \text{Gr }U(.) = \{ (t,u) \ : \ (t,u) \in \text{Gr } U(.) \text{ and } \Psi(t,u) > 0  \}, \]
and denote by $E_t$ the $t-$section of $E$, i.e.
\[ E_t := \{   t \in [0,T] \ : \ (t,u) \in E  \} .   \]
Then, $E_t := \cup_{j \ge 1} B_j$.

\noindent
Now assume, by contradiction, that (ii) of Theorem \ref{atomic} is violated. Therefore, meas$(E_t) >0$. Write $\delta:=\text{meas}(E_t).$ Since, $\text{meas}(E_t) = \lim\limits_{j \to \infty} \text{meas}(B_j)$, there exists $j_0\in \N$ such that $\text{meas}(B_j) \ge \frac{\delta}{2}$ for all $j \ge j_0$. 
Therefore, for all $t \in B_{j_0}$, there exists $u_t \in U(t)$ such that $\Psi(t,u_t) \ge \frac{1}{j_0}$. 
Take $i_0 \in \N$ such that
\begin{equation} \label{inequality measure} 2c M_p \sum_{k \ge N_{i_0}} \mu(\omega_k) \le \frac{1}{8} \frac{1}{j_0} \ , \end{equation}
(here $c>0$ is the upper bound for $|f|$ (see (A3)$'$) and  $M_p>0$ is an upper bound for $||p( ., \omega)||_{L^\infty}$), and
\begin{equation} \label{choice of rho}  \rho'_{i_0} \le \min \bigg\{ \frac{\delta}{8} \  ; \ \frac{1}{16j_0} \ ; \ \frac{ \delta}{32 j_0  \beta M_p \int_{0}^{T} k_f(s) \ ds}   \bigg\}  . \end{equation}
(Recall that $\beta>0$ is the number provided by Lemma \ref{lemma_continuity_omega_control} (i) and $\{\rho_i'\}$ is the decreasing sequence appearing in Step 2 of the proof of Theorem \ref{theorem_nco_general_case_probability_space}.)

\noindent
For all $i \ge i_0$ and for all $t \in B_{j_0} \cap A_{\rho_i}$, we have

\begin{align} \label{inequality}
\nonumber\frac{1}{j_0} & \leq \sum_{k \ge 0} p(t,\omega_k) \cdot \bigg[f(t,\bar{x}(t,\omega_k),u_t,\omega_k) -  f(t,\bar{x}(t,\omega_k),\bar u(t),\omega_k)     \bigg] \mu(\omega_k) \\ & \nonumber = \sum_{k=0}^{N_i} p_i(t,\omega_k) \cdot \bigg[f(t,{x_i}(t,\omega_k),u_t,\omega_k) -  f(t,{x_i}(t,\omega_k),\bar u(t),\omega_k)     \bigg] \mu(\omega_k) \\ & \nonumber + \sum_{k=0}^{N_i} \bigg( p(t,\omega_k) \cdot \bigg[f(t,\bar{x}(t,\omega_k),u_t,\omega_k) -  f(t,\bar{x}(t,\omega_k),\bar u(t),\omega_k)     \bigg] \\ & - p_i(t,\omega_k) \cdot \bigg[f(t,{x_i}(t,\omega_k),u_t,\omega_k) -  f(t,{x_i}(t,\omega_k),\bar u(t),\omega_k)     \bigg] \bigg) \mu(\omega_k) \nonumber \\ &  + \sum_{k=N_i+1}^{+ \infty} \mu(\omega_k) p(t,\omega_k) \cdot \big[ f(t,\bar x(t,\omega_k),u,\omega_k) - f(t,\bar x(t,\omega_k),\bar u(t),\omega_k)  \big] \ .
\end{align}
\noindent
Condition \ref{a3} established in Step 2 of Theorem \ref{theorem_nco_general_case_probability_space} proof implies that the first term on the right-hand side of (\ref{inequality}) satisfies
\[ \sum_{k=0}^{N_i} p_i(t,\omega_k) \cdot \bigg[f(t,{x_i}(t,\omega_k),u_t,\omega_k) -  f(t,{x_i}(t,\omega_k),\bar u(t),\omega_k)     \bigg] \mu(\omega_k) \leq  \rho_i' \quad \text{for all } t \in A_{\rho_i'}.  \]
Concerning the second term on the right-hand side of (\ref{inequality}) we make use of the boundedness of $f$ and $\|p(., \omega) \|_{L^\infty}$, and the estimate  (\ref{xi}): we obtain
\begin{align*}
S :=  &\sum_{k=0}^{N_i}  \bigg( p(t,\omega_k) \cdot \bigg[f(t,\bar{x}(t,\omega_k),u_t,\omega_k) -  f(t,\bar{x}(t,\omega_k),\bar u(t),\omega_k)     \bigg]  \\ 
& \quad - p_i(t,\omega_k) \cdot \bigg[f(t,{x_i}(t,\omega_k),u_t,\omega_k) -  f(t,{x_i}(t,\omega_k),\bar u(t),\omega_k)     \bigg] \bigg)  \mu(\omega_k) \\ 
 \leq &  2c \sum_{k=0}^{N_{i_0}} \mu(\omega_k)|p(t,\omega_k)-p_i(t,\omega_k)| +  2c \sum_{k=N_{i_0}+1}^{N_i} \mu(\omega_k)  |p(t,\omega_k)-p_i(t,\omega_k)| +   2 k_f(t) \beta M_p \rho_i' \times  \left( \sum_{k=0}^{N_i}  \mu(\omega_k)\right) .
\end{align*}
Take $i_1 \ge i_o$ large enough such that for all $i \ge i_1$
\[ \| p(.,\omega_k)   - p_i(., \omega_k) \|_{L^\infty}  \le \frac{1}{16c} \frac{1}{j_0} \quad \text{for all } k=0, \ldots, N_{i_0}.  \]
Therefore, owing to the choice made in (\ref{inequality measure}), we have $S\le  2 k_f(t) \beta M_p \rho_i' + \frac{1}{8j_0} + \frac{1}{4j_0}$.
Then, from (\ref{inequality}), we obtain that
\[   \frac{1}{2j_0}  \leq \rho_i' \bigg[ 1 +  2 k_f(t) \beta M_p \bigg]  . \]
By integrating over the measurable set $B_{j_0} \cap A_{\rho_i'}$, taking into account that $ \frac{\delta}{2} \le \mbox{meas}(B_{j_0})\le \delta$ and meas$([0,T] \setminus A_{\rho'_i}) \le \rho_i'$, we arrive at 
$$ 
\frac{1}{8} \frac{\delta}{j_0} \le \frac{1}{16} \frac{\delta}{j_0},
$$
a contradiction. Therefore, also the maximality condition (\ref{item: maximality condition atomic case}) of Theorem \ref{atomic} holds true.

\qed

\ \\
\noindent
{\bf Proof of Theorem \ref{proposition: smooth data nco}}. A scrutiny of Theorem \ref{theorem_nco_general_case_probability_space} proof reveals that Steps 1, 2 and 3 are applicable providing a simplified result. 
Indeed, taking into account hypotheses (\ref{item: c1 assumption smooth data proposition nco})-(\ref{item: c2 assumption smooth data proposition nco}) on $f(t,.,u,\omega)$ and $g(.,\omega)$, we obtain a family of costate arcs $\widetilde p(.,\omega)$, for $\omega\in \widehat \Omega$ ($\widehat \Omega$ is a countable dense subset of $\mbox{supp}(\mu)$), satisfying the properties listed at the end of the Step 3 of the proof of Theorem \ref{theorem_nco_general_case_probability_space}, where (\ref{equation: adjoint system proof general case}) and (\ref{equation: transversality condition proof general case}) read now as
\begin{equation}\label{1}
-\dot {\widetilde p}(t, \omega) = [\nabla_x f(t, \bar x (t, \omega), \bar u(t), \omega)]^T   \widetilde p(t, \omega)  \quad  \mbox{a.e.} \;\;\; t\in [0,T]\ ,
\end{equation} 
and
\begin{equation}\label{2}
- \widetilde p(T, \omega)  =  \nabla_x g (\bar{x}(T,\omega)),\omega)	
\end{equation} 
for all $\omega \in \hat \Omega$. 
Notice, that the multiplier $\lambda$ cannot take the value $0$, for otherwise we would obtain a contradiction with the  nontriviality condition. 
Then, normalizing we can take $\lambda=1$.

\noindent
We claim now that we can extend in a unique way the family of arcs $\widetilde p (.,\omega)$, for $\omega\in \widehat \Omega$, to a $\mathcal{L}\times\mathcal{B}_\Omega$ measurable function $p(.,.): [0,T] \times \Omega \rightarrow \R^n$ such that for all $\omega \in \mbox{supp}(\mu) $ we have:
\begin{enumerate}[label=(\roman*)$''$, ref=(\roman*)$''$]
	\item $p(., \omega) \in W^{1,1}([0,T],\R^n)$;
	\item \label{item: adjoint system condition proof corollary} $-\dot p(t,  \omega) = [\nabla_x f(t, \bar x (t,   \omega), \bar u(t),   \omega)]^T    p(t,  \omega) $\quad  \ a.e. $t\in [0,T]$ ;
	\item  \label{item: transversality condition proof corollary} $- p(T,   \omega)  =  \nabla_x g (\bar{x}(T,  \omega); \omega)$.
\end{enumerate}
Indeed, take any  $\omega \in \Omega \setminus \widehat \Omega$. If $\omega\in \Omega \setminus \mbox{supp}(\mu)$ we set $p(.,\omega)=0$. So we continue the analysis considering the case $\omega \in \mbox{supp}(\mu)\setminus \widehat \Omega$. 
Then, since $\widehat \Omega$ is dense in $\mbox{supp}(\mu)$, there exists a sequence $\{\widehat \omega _i\} \subset \widehat{\Omega}$ converging to $\omega$. 
Assumptions \ref{A2'_nco_general_case} and \ref{A4'_nco_general_case} guarantee that $|\nabla_x f(t, \bar x (t,   \omega), \bar u(t),  \widehat \omega _i)|\le k_f(t)$ a.e. $t\in [0,T]$ and $|\nabla_x g|\le k_g$. From (\ref{1}) we deduce that $\{ \dot{\widetilde p}(.,\widehat \omega _i)  \}$ is uniformly integrally bounded, and (\ref{2}) guarantees that $|\widetilde p(T, \widehat \omega _i)|\le k_g$.
Then, by a standard compactness argument, taking a subsequence (we do not relabel), there exists $p(.,\omega) \in W^{1,1}([0,T],\R^n)$ such that
\[  \widetilde p(t,\widehat \omega _i) \to p(t, \omega) \qquad \text{uniformly on } [0,T] \text{  as  } i \to \infty   \]
\[ \dot{\widetilde p}(t,\widehat \omega _i)  \rightharpoonup \dot{p}(t,\omega)  \qquad \text{weakly in } L^1   \] 
and
\begin{equation}\label{limit1}
  - \dot{p}(t,\omega)   = [\nabla_x f (t,\bar x(t,\omega),\bar u(t), \omega)]^T p(t,\omega)   \quad  \text{a.e. } t \in [0,T] \ ,
\end{equation} 
\begin{equation}\label{limit2}
-p(T,\omega) = \nabla_x g(\bar x(T,\omega);\omega)	\ .
\end{equation} 
(The last two equalities are a consequence of Lemma \ref{lemma_continuity_omega_control} (ii).)

\noindent
This, being true for any sequence $\{\widehat \omega _i\}\subset \widehat \Omega$ converging to $\omega \in \mbox{supp}(\mu) \setminus \widehat \Omega$, since the limit arc satisfies the same conditions (\ref{limit1})-(\ref{limit2}), we conclude that we can extend the family of arcs $\widetilde p (.,\omega)$ simply taking the limit:
\begin{equation}
\label{choice of adjoint arc} p(.,\omega) := \lim\limits_{\rho_{\Omega}(\omega,\widehat \omega) \to 0, \;\; \widehat \omega \in \widehat \Omega} \widetilde p(.,\widehat \omega)\ ,
\end{equation}
confirming the claim above. It remains to prove the Weierstrass condition \ref{item: maximality condition corollary}.  We follow exactly the same analysis of Step 4 of Theorem \ref{theorem_nco_general_case_probability_space} proof, taking now the simplified version of the definition of the set $D$ in which we take into account the regularity of functions $f$ and $g$, the fact that $\lambda=1$ and we do not have end-point constraints:
\begin{align*} D := \{ (\omega,\xi) & \in \Omega \times \R^K \ | \ \omega \in \Omega \; \textrm{ and }  \xi = \big( \Psi_k(p(.,\omega), \omega) \big) _{k=1,\ldots,K} \textrm{ for some } \mathcal{L}\times\mathcal{B}_\Omega \text{ measurable function } \\ & p: [0,T] \times \Omega \rightarrow \R^n \text{ such that }  \ \ p(.,\omega) \in \mathcal{P}_S(\omega) \ \text{ for all }\omega \in \mbox{supp}(\mu) \} 
\end{align*}
where now, we set
\begin{align*} \mathcal{P}_S(\omega) :=  & \Bigg\{ q(.,\omega) \in  W^{1,1}( [0,T], \R^n) \ : \  
-q(T,\omega) = \nabla_x g(\bar x(T,\omega);\omega) \\ &  - \dot q(t,\omega) = [\nabla_x f (t,\bar x(t,\omega),\bar u(t), \omega)]^T  q(t,\omega) \ \text{ a.e. } t \in [0,T] \Bigg\}. \end{align*}
\noindent
The uniqueness of solutions to systems appearing in $\mathcal{P}_S(\omega)$ allows to conclude.

\qed

\ \\

\noindent
{\bf Acknowledgements.} The authors are thankful to Richard B. Vinter for his suggestion to study necessary conditions for average cost optimal control problems, and to the referees for their many helpful comments.

\end{document}